\documentclass[letterpaper, 12pt]{article}
\usepackage{palatino}

\title{Actions in the Airy line ensemble and convergence to the Airy sheet}
\author{B\'alint Vir\'ag\thanks{\textsc{Dept. of Mathematics, University of Toronto}. \texttt{balint@math.toronto.edu}} \and
Xuan Wu\thanks{\textsc{Dept. of Mathematics, University of Illinois Urbana-Champaign}. \texttt{xuanzw@illinois.edu}}}
\date{}
\usepackage[letterpaper, margin=.9 in]{geometry}

\usepackage{amsmath, amsthm, amssymb}
\usepackage{enumerate}
\usepackage{bbm}
\usepackage{color}

\usepackage[longnamesfirst]{natbib}
\bibpunct[ ]{(}{)}{,}{a}{}{,}

\RequirePackage[colorlinks,urlcolor=blue,linkcolor=black,citecolor=blue]{hyperref}
    \newtheorem{theorem}{Theorem}
    \newtheorem{lemma}[theorem]{Lemma}
    \newtheorem{proposition}[theorem]{Proposition}
    \newtheorem{corollary}[theorem]{Corollary}

\theoremstyle{definition} 
    \newtheorem{definition}[theorem]{Definition}
    \newtheorem{remark}[theorem]{Remark}

\numberwithin{theorem}{section}
\numberwithin{equation}{section}

\newcommand{\eps}{\varepsilon}

\newcommand{\ed}{\stackrel{d}{=}}

\newcommand{\cA}{{\mathcal A}}

\newcommand{\cS}{{\mathcal S}}

\newcommand{\uc}{\text{\sc uc}}

\newcommand{\ar}[1][4pt]{\mathrel{%
   \hbox{\rule[\dimexpr\fontdimen22\textfont2-.2pt\relax]{#1}{.4pt}}%
   \mkern-4mu\hbox{\usefont{U}{lasy}{m}{n}\symbol{41}}}}

\begin{document}
\maketitle
\begin{abstract}
Actions in the Airy line ensemble represent distances from an infinitely far object. We characterize the Airy sheet by $$\mathcal \cS(x,\cdot)=T^x(\cdot,1),$$ 
where $T^x$ is the unique action in the Airy line ensemble satisfying a growth condition depending on $x$. This provides a new simple framework for establishing convergence to the Airy sheet. We present simple conceptual proofs of such results in the case of Brownian last passage, the O'Connell-Yor semidiscrete polymer, the log-gamma polymer and the KPZ equation. 
\end{abstract}

\section{Introduction}
The Airy line ensemble is a two-parameter scaling limit of many integrable planar multipath last passage percolation models. In the KPZ universality class, convergence to the Airy line ensemble was first established for the PNG model \cite{prahofer2002scale} and TASEP \cite{johansson2000shape,borodin2007asymptotics}. The KPZ universality conjecture predicts its emergence as a limit of most natural polymer, first-passage, and last-passage models, even in non-integrable cases. Recently, \cite{aggarwal2025strong} showed that the Airy line ensemble is uniquely determined by a Gibbs resampling property introduced by \cite{CH}, inherited from finite integrable models. Combined with case-by-case proofs of tightness, this yields convergence to the Airy line ensemble in broad settings.

The goal of this paper is to provide a direct and transparent path from convergence to the Airy line ensemble to convergence to the Airy sheet. This is the main step in showing convergence the full scaling limit, the directed landscape, which has independent Airy sheet increments. Prior approaches relied on intricate analysis of pre-limiting Busemann functions and their convergence, see \cite{DOV,DV, wu2023kpz}. Here, we bypass such  analysis by introducing {\bf actions} in the Airy line ensemble: solutions to a semi-discrete Hamilton-Jacobi equation encoding distances from points at infinity.

For $x > 0$, the Airy sheet $\mathcal{S}(x,y)$ is the distance from $(y,1)$ to a point at infinity in direction $x$, also known as a Busemann function \cite{DOV}. These extend naturally to the full domain $\mathbb{R} \times \mathbb{N}$ of the Airy line ensemble and satisfy {\bf action recursion}:
\begin{equation}\label{e:action-intro}
T(x',k) = \bigl(T(x,k) + \cA(x',k) - \cA(x,k)\bigr) \vee \sup_{z \in [x,x']} \bigl(T(z,k+1) + \cA(x',k) - \cA(z,k)\bigr),
\end{equation}
For all  $x<x'\in \mathbb R$ and  $k\in \mathbb N$. (This is simply a local equation characterizing semidiscrete last passage percolation values in the environment $\mathcal A$: $T(\cdot,k)$ is the Skorokhod reflection of $\mathcal A(\cdot,k)$ from $T(\cdot, k+1)$.)

We call functions $T$ satisfying \eqref{e:action-intro} $\mathcal A$-\textbf{actions}. While not all actions are Busemann functions, they are characterized by a simple  growth condition. This yields a new definition: $\mathcal{S}(x,\cdot)$ is the value at level $k=1$ of the unique action $T$ satisfying:
\begin{equation}\label{e:growth}
\liminf_{y\to\infty} \frac{T(y,1)+y^2}{y}\le 2x, \qquad \liminf_{y\to-\infty} \frac{T(y,1)+y^2}{-y}\le -2x.
\end{equation}

The most important conceptual part of the proof is the following theorem, which identifies all top lines of $\mathcal A$-actions with KPZ evolutions from all initial conditions on $[0,\infty)$. Let $\uc_+$ denote the set of upper-semicontinuous functions from $[0,\infty)\to [-\infty,\infty)$.

\begin{theorem}\label{t:main}
Almost surely, the following identification holds:
\begin{align*}
\{T(1,\cdot):T\mbox{ is an }\mathcal{A}\mbox{-action and }T(y,1)\in\mathbb{R}\ \forall y \} 
&= \{\mathcal{S}(f,\cdot):f\in \uc_+,\;\mathcal{S}(f,y)\in \mathbb{R}\ \forall y \},
\end{align*}
where \begin{equation}\label{e:ASvari}
\mathcal{S}(f,y) = \sup_{x \ge 0} [f(x) + \mathcal{S}(x,y)].
\end{equation}
\end{theorem}

This characterization is the cornerstone of our convergence framework. We derive two practical corollaries that characterize Airy sheet values.

\begin{corollary}\label{c:main}
Let $T$ be an $\mathcal{A}$-action with $T(y,1)\in\mathbb{R}$ for all $y$. Suppose there exists $x\in\mathbb{R}$ such that the growth condition \eqref{e:growth} holds.
Then there exists a random $C$ such that $T(y,1)=\mathcal{S}(x,y)+C$ for all $y\in\mathbb{R}$.
\end{corollary}

\begin{corollary}\label{c:main2}
Let $T$ be an $\mathcal{A}$-action. Suppose there exists $x\in\mathbb{R}$ such that for all $n\in \mathbb{Z}$, $T(n,1)+(n-x)^2$ has the GUE Tracy-Widom distribution. Then almost surely $T(y,1)=\mathcal{S}(x,y)$ for all $y\in \mathbb{R}$.
\end{corollary}

We introduce {\bf action representations}: functions satisfying the action recursion, quadrangle inequalities, and symmetry. We prove (Theorems~\ref{thm:action-rep-sym}, \ref{t:action-rep_discrete}) that convergence in distribution to the Airy line ensemble, combined with Tracy-Widom marginals, implies local uniform convergence to the Airy sheet. This framework yields short, unified convergence proofs for:

\begin{itemize}
\item Brownian last passage percolation recovering \citet{DOV},
\item O'Connell-Yor semidiscrete polymers of \cite{o'connell2012DP},
\item Log-gamma directed polymers of \cite{sep-palv-2009}, recovering \cite{zhang2025convergence},
\item The KPZ equation, recovering \cite{wu2023kpz}.
\end{itemize}

Our approach eliminates case-by-case Busemann analysis, offering a conceptually clear, modular, and broadly applicable method for Airy sheet convergence in the KPZ class.

The paper is organized as follows. In Section~\ref{sec:basics}, we define actions, establish their connection to infinite geodesics, and prove a variational decomposition (Proposition~\ref{prop:decomposition}). Subsection~\ref{sec:conv_to_AS} introduces action representations and proves general convergence theorems to the Airy sheet (Theorems~\ref{thm:action-rep-sym} and \ref{t:action-rep_discrete}). The core identification of actions with Airy sheet values is established in Section~\ref{sec:Airyaction}, where we prove Theorem~\ref{t:main} and Corollaries~\ref{c:main} and \ref{c:main2}. Applications are presented in Section~\ref{sec:application}: Brownian last passage percolation (Theorem~\ref{t:BLPP_to_FP}), O'Connell-Yor polymers (Theorem~\ref{t:OY_to_FP}), log-gamma polymers (Theorem~\ref{t:LG_to_FP}), and the KPZ equation (Theorems~\ref{thm:KPZ_action}). An appendix reviews the hypograph and Skorokhod topologies used in compactness arguments.

\section{Basics of actions and geodesics}\label{sec:basics}

The goal of this section is to prove Proposition~\ref{prop:decomposition}, which expresses the action values in a variational form involving infinite geodesics.

Subsection~\ref{sec:action} introduces actions, the core concept of this paper. ``Action" is a standard term in Lagrangian mechanics. It is usually denoted $S$, but that clashes with our notation for the Airy sheet, so we use the notation $T$. In Subsection~\ref{sec:LPP_geodesic}, we define last passage percolation (LPP) and geodesics. Subsection~\ref{sec:connection} establishes the connection between infinite geodesics and actions. This connection is standard, but we are not aware of any references that would work in the present setting. 

Throughout this paper, the underlying space $M$ is defined as
\begin{equation}\label{e:M}
M = I \times J, \qquad
I = [\alpha, \infty) \quad \text{or} \quad I = (-\infty, \infty),
\qquad
J = \{ k \in \mathbb{N} : k < j \}
\end{equation}
for some $\alpha \in \mathbb{R}$ and $j \in \mathbb{N} \cup \{\infty\}$. This accommodates both pre-limiting and limiting models. An {\bf environment} is a function $A: M \to \mathbb{R}$ such that for each $k \in J$, the map $y \mapsto A(y,k)$ is cadlag in $y$.

\subsection{Actions}\label{sec:action}

\begin{definition}[Actions]
Let $A$ be an environment on $M=I\times J$, see \eqref{e:M}. An $A$-\textbf{action} is a function $T:M\to \mathbb{R}\cup\{\pm\infty\}$ such that for all $k\in J$ and $x,y\in I$ with $x<y$, we have
\begin{align}\label{equ:Local_metric_composition_law}
T(y,k) =(T(x,k)+A(y,k)-A(x,k) ) \vee \sup_{z\in [x,y]}\big( T(z,k+1)+A(y,k)-A(z,k)\big).
\end{align}
Here we adopt the convention that $T(z,k+1)\equiv -\infty$ if $k+1\notin J$. Equation \eqref{equ:Local_metric_composition_law} is referred to as the {\bf action recursion}. 
\end{definition}

In the next few lemmas we establish some basic topological properties of actions. We first show that the supremum of actions is also an action.

\begin{lemma}\label{lem:actionsup}
Let $A$ be an environment on $M=I\times J$ and let $T_\lambda$, $\lambda\in\Lambda$ be a family of $A$-actions. Then $T=\sup_{\lambda\in\Lambda} T_\lambda$ is an $A$-action.
\end{lemma}
\begin{proof}
Fix $x<y$ in $I$ and $k\in J$. For each $\lambda\in\Lambda$, using $T_\lambda\leq T$, we have
\begin{align*}
T_\lambda(y,k) =&(T_\lambda(x,k)+A(y,k)-A(x,k) ) \vee \sup_{z\in [x,y]}\big( T_\lambda(z,k+1)+A(y,k)-A(z,k)\big)\\
\leq&(T(x,k)+A(y,k)-A(x,k) ) \vee \sup_{z\in [x,y]}\big( T(z,k+1)+A(y,k)-A(z,k)\big).
\end{align*}
Taking supremum over $\lambda\in\Lambda$ gives
\begin{align*}
T(y,k)\leq (T(x,k)+A(y,k)-A(x,k) ) \vee \sup_{z\in [x,y]}\big( T(z,k+1)+A(y,k)-A(z,k)\big).
\end{align*}

For each $\lambda\in\Lambda$, using $T_\lambda\leq T$ we have $
T_\lambda(x,k)+A(y,k)-A(x,k)\leq T_\lambda(y,k)\leq T(y,k).$  This gives
\begin{align*}
T(x,k)+A(y,k)-A(x,k)\leq T(y,k).
\end{align*}
Similarly, for any $z\in [x,y]$ we have
\begin{align*}
T(z,k+1)+A(y,k)-A(x,k)\leq T(y,k).
\end{align*}
Combining the above yields
\begin{align*}
(T(x,k)+A(y,k)-A(x,k) ) \vee \sup_{z\in [x,y]}\big( T(z,k+1)+A(y,k)-A(z,k)\big)\leq T(y,k).
\end{align*}
The proof is complete.
\end{proof}

Next, we show that the upper semicontinuous version of an action is also an action. 

\begin{lemma}\label{lem:UCaction}
Let $A$ be a continuous environment on $M=I\times J$ and let $T$ be an $A$-action. Denote by $T'$ the upper semicontinuous version of $T$. That is,
\begin{align*}
T'(x,k)=\lim_{\varepsilon\to 0}\sup_{w\in [x-\varepsilon,x+\varepsilon]\cap I} T(w,k).
\end{align*}
Then $T'$ is an $A$-action.
\end{lemma}
\begin{proof}
We aim to verify the action recursion \eqref{equ:Local_metric_composition_law}. For any $0<\varepsilon<1$, let 
\begin{align*}
T^\varepsilon(x,k)=\sup_{w\in [x-\varepsilon,x+\varepsilon]\cap I} T(w,k).
\end{align*}
From now on, we fix $x,y\in I$ with $x<y$ and $k\in J$. We further assume $k+1\in J$. The argument for the case $k+1\notin J$ is similar and simpler. Let 
$$\delta(\varepsilon)= \sup\{ |A(w,k)-A(w',k)|\vee |A(w,k+1)-A(w',k+1)|:\,  w,w'\in [x-1,y+1]\cap I, |w-w'|\leq \varepsilon \}. $$
From the continuity of $A$, $\lim_{\varepsilon\to 0}\delta(\varepsilon)=0$. We assume $\varepsilon<y-x$. We claim that
\begin{align}\label{equ:UC_action_1}
(T^\varepsilon(x,k)+A(y,k)-A(x,k) ) \vee \sup_{z\in [x,y]}\big( T^\varepsilon(z,k+1)+A(y,k)-A(z,k)\big)\leq T^\varepsilon(y,k)+2\delta(\varepsilon) 
\end{align}
and that
\begin{align}\label{equ:UC_action_2}
T^\varepsilon(y,k)\leq &(T^\varepsilon(x,k)+A(y,k)-A(x,k) ) \vee \sup_{z\in [x,y]}\big( T^\varepsilon(z,k+1)+A(y,k)-A(z,k)\big)+\delta(\varepsilon).
\end{align}
Sending $\varepsilon\to 0$ in \eqref{equ:UC_action_1} and \eqref{equ:UC_action_2} verifies \eqref{equ:Local_metric_composition_law} for $T'$.

It remains to prove \eqref{equ:UC_action_1} and \eqref{equ:UC_action_2}. We begin with \eqref{equ:UC_action_1}. For any $w\in [x-\varepsilon,x+\varepsilon]\cap I$, we have
\begin{align*}
T(w,k)+A(y,k)-A(x,k)\leq &T(w,k)+A(y,k)-A(w,k)+\delta(\varepsilon)\\
\leq &T(y,k)+\delta(\varepsilon)
\leq T^\varepsilon(y,k)+\delta(\varepsilon).
\end{align*}
Hence
\begin{align}\label{equ:UC_action_3}
T^\varepsilon(x,k)+A(y,k)-A(x,k)\leq T^\varepsilon(y,k)+\delta(\varepsilon).
\end{align}
Fix $z\in [x,y]$. For any $w\in [z-\varepsilon,z+\varepsilon]\cap I$,
\begin{align*}
T(w,k+1)+A(y,k)-A(z,k)\leq &T(w,k+1)+A(y+\varepsilon,k)-A(w,k)+2\delta(\varepsilon)\\
\leq &T(y+\varepsilon,k)+2\delta(\varepsilon)\leq T^\varepsilon(y,k)+2\delta(\varepsilon).
\end{align*}
This implies
\begin{align*}
T^\varepsilon(z,k+1)+A(y,k)-A(z,k)\leq T^\varepsilon(y,k)+2\delta(\varepsilon),
\end{align*}
and
\begin{align}\label{equ:UC_action_4}
\sup_{z\in [x,y]} \big(T^\varepsilon(z,k+1)+A(y,k)-A(z,k)\big)\leq T^\varepsilon(y,k)+2\delta(\varepsilon).
\end{align}
Combining \eqref{equ:UC_action_3} and \eqref{equ:UC_action_4} yields \eqref{equ:UC_action_1}.

Next, we turn to \eqref{equ:UC_action_2}. Fix $w\in [y-\varepsilon,y+\varepsilon]\cap I$. Then
\begin{equation}\label{equ:UC_action_5}
\begin{split}
T(x,k)+A(w,k)-A(x,k)\leq &T(x,k)+A(y,k)-A(x,k)+\delta(\varepsilon)\\
\leq &T^\varepsilon (x,k)+A(y,k)-A(x,k)+\delta(\varepsilon).
\end{split}
\end{equation}
We claim that
\begin{align}\label{equ:UC_action_6}
\begin{split}
\sup_{z\in [x,w]}&\big( T(z,k+1)+A(w,k)-A(z,k)\big)
\\ &\leq \sup_{z\in [x,y]} \big(T^\varepsilon(z,k+1)+A(y,k)-A(z,k)\big)+\delta(\varepsilon). 
\end{split}
\end{align}
Assume for a moment \eqref{equ:UC_action_6} holds. Combining \eqref{equ:UC_action_5} and \eqref{equ:UC_action_6} gives
\begin{align*}
T(w,k) =&(T(x,k)+A(w,k)-A(x,k) ) \vee \sup_{z\in [x,w]}\big( T(z,k+1)+A(w,k)-A(z,k)\big)\\
\leq &(T^\varepsilon(x,k)+A(y,k)-A(x,k) ) \vee \sup_{z\in [x,y]}\big( T^\varepsilon(z,k+1)+A(y,k)-A(z,k)\big)+\delta(\varepsilon).
\end{align*}
Taking the supremum over $w\in [y-\varepsilon,y+\varepsilon]$ gives \eqref{equ:UC_action_2}.

Lastly, we prove \eqref{equ:UC_action_6}. Fix $z_0\in [x,w]$. If $z_0\in [x,y]$, then
\begin{align*}
T(z_0,k+1)+A(w,k)-A(z_0,k)\leq &T(z_0,k+1)+A(y,k)-A(z_0,k)+\delta(\varepsilon)\\
\leq &\sup_{z\in [x,y]} \big(T(z,k+1)+A(y,k)-A(z,k)\big)+\delta(\varepsilon)\\
\leq &\sup_{z\in [x,y]} \big(T^\varepsilon(z,k+1)+A(y,k)-A(z,k)\big)+\delta(\varepsilon).
\end{align*}
If $z_0>y$. Then $y<z_0\leq w\leq y+\varepsilon$. In particular, $|w-z_0|\leq \varepsilon$. Hence
\begin{align*}
T(z_0,k+1)+A(w,k)-A(z_0,k)\leq &T^\varepsilon(y,k+1)+\delta(\varepsilon)\\
\leq &\sup_{z\in [x,y]} \big(T^\varepsilon(z,k+1)+A(y,k)-A(z,k)\big)+\delta(\varepsilon).
\end{align*}
Combining the above gives \eqref{equ:UC_action_6}. The proof is complete. 
\end{proof}

\subsection{Last passage percolation and geodesics}\label{sec:LPP_geodesic}

In this subsection, we establish the relationship between semidiscrete last passage percolation, its geodesics, and actions. In this and the following subsections, we always assume the environment $A$ is continuous. We equip the extended natural numbers $\mathbb{N}\cup\{\infty\}$ with the topology which compactifies $\mathbb{N}$ (with the discrete topology) at $\infty$. For $(x_1,k_1), (x_2,k_2)\in \mathbb{R}\times(\mathbb{N}\cup\{\infty\})$, we let 
$$(x_1,k_1)\preceq (x_2,k_2) \mbox{ denote the relation }x_1\leq x_2,\mbox{ and } k_1\geq k_2.$$

\begin{definition}
A \textbf{directed path} $\gamma$ is a closed, non-empty subset of $\mathbb{R}\times(\mathbb{N}\cup\{\infty\})$ such that $(\gamma,\preceq)$ is an ordered set, and the coordinate projections $\Pi_1(\gamma)\subset\mathbb{R}$ and $\Pi_2(\gamma)\subset\mathbb{N}\cup\{\infty\}$ are closed intervals.

Given $(x_1,k_1), (x_2,k_2)\in\mathbb{R}\times\mathbb{N}$, we denote by $\mathcal{P}((x_1,k_1)\ar (x_2,k_2))$ the collection of directed paths $\gamma$ that satisfy $(x_1,k_1),(x_2,k_2)\in\gamma$ and $(x_1,k_1)\preceq q\preceq (x_2,k_2)$ for all $q\in \gamma$. Given $(x,k)\in \mathbb{R}\times\mathbb{N}$, we denote by $\mathcal{P}((-\infty,\ast)\ar (x,k))$ the collection of directed paths $\gamma$ such that $(x,k)\in\gamma$, $q\preceq (x,k)$ for all $q\in \gamma$, and $\Pi_1(\gamma)=(-\infty,x]$.  
\end{definition}

We associate each directed path $\gamma\in \mathcal{P}((x_1,k_1)\ar (x_2,k_2))$ a vector $(z_{k_1+1}, z_{k_1},\dots, z_{k_2})$ through $z_{k_1+1}=x_1, z_{k_2}=x_2$, and
\begin{align*} 
z_j= \min\{ z\in [x_1,x_2]\, :\, (z,j-1)\in \gamma  \}\ \textup{for}\ k_2+1\leq  j\leq k_1. 
\end{align*}
Given a continuous environment $A$ on $M=I\times J$, $(x_1,k_1),(x_2,k_2)\in M$ with $(x_1,k_1)\preceq (x_2,k_2)$, and $\gamma\in \mathcal{P}((x_1,k_1)\ar (x_2,k_2))$, the length of $\gamma$ is defined to be 
\begin{align}\label{equ:path_length} 
L(\gamma,A)=\sum_{j=k_2}^{k_1} A(z_j,j)-A(z_{j+1},j).
\end{align}  
For $q_1\preceq q_2$ in $M$, we define the \textbf{last passage time} from $q_1$ to $q_2$ by
\begin{align}\label{def:LPPtime}
A(q_1\ar q_2)=\max_{\gamma\in\mathcal{P}(q_1\ar q_2)} L(\gamma,A).
\end{align}
It is direct to check that we have the triangle inequality
\begin{equation}\label{equ:triangle}
A(q_1\ar q_2)+A(q_2\ar q_3)\leq A(q_1\ar q_3)
\end{equation}
for all $q_1\preceq q_2\preceq q_3$ in $M$. Also, we have the following quadrangle inequality.
\begin{lemma}[Proposition 3.8 in \cite{DOV}]\label{lem:quadrangle_LP}
Let $A$ be a continuous environment on $M=I\times J$. For any $x\leq x'\leq y\leq y'$ in $I$ and $k\leq m$ in $J$, we have
\begin{equation}\label{equ:quadrangle_LP}
A((x,m)\ar (y,k))+A((x',m)\ar (y',k))\geq A((x,m)\ar (y',k))+A((x',m)\ar (y,k)).
\end{equation}
\end{lemma}
A directed path $\gamma$ is called a \textbf{geodesic} if for all $q_1\preceq q_2\preceq q_3$ in $\gamma\cap (\mathbb{R}\times\mathbb{N})$, we have
\begin{equation*} 
A(q_1\ar q_2)+A(q_2\ar q_3)=A(q_1\ar q_3).
\end{equation*}
We note that for any $q_1\preceq q_2$ in $M$, there exists a geodesic $\gamma\in \mathcal{P}(q_1\ar q_2)$ such that $L(\gamma,A)=A(q_1\ar q_2)$.

\begin{lemma}\label{lem:Action+LPP}
Let $A$ be a continuous environment on $M$ and $T$ be an $A$-action. Then for any $q_1\preceq q_2$ in $M$, we have
\begin{equation}\label{equ:Action+LPP}
T(q_2)\geq T(q_1)+A(q_1\ar q_2).
\end{equation}
\end{lemma}
\begin{proof}
Let $q_1=(x_1,k_1)$ and $q_2=(x_2,k_2)$. We aim to show that
\begin{equation}\label{equ:Action+LPP_1}
T(x_2,k_2)\geq T(x_1,k_1)+A((x_1,k_1)\ar (x_2,k_2)).
\end{equation}
We use an induction argument on $k_1-k_2$. Suppose $k_1-k_2=0$. Then \eqref{equ:Action+LPP_1} follows from \eqref{equ:Local_metric_composition_law} with $A((x_1,k_2)\ar (x_2,k_2))=A(x_2,k_2)-A(x_1,k_2)$. 

Suppose $k_1=k_2+1$. Let $\gamma\in \mathcal{P}((x_1,k_2+1)\ar (x_2,k_2))$ be a geodesic. Set 
$$
z'=\min\{z\in [x_1,x_2]\, :\, (z,k_2)\in\gamma \}.
$$
Then 
$$
A((x_1,k_2+1)\ar (x_2,k_2))=A((x_1,k_2+1)\ar (z',k_2+1))+A((z',k_2)\ar (x_2,k_2)).
$$
Applying \eqref{equ:Local_metric_composition_law} twice, we have
\begin{align*}
T(x_2,k_2)\geq &T(z',k_2+1)+A((z',k_2)\ar (x_2,k_2))\\
\geq& T(x_1,k_2+1)+A((x_1,k_2+1)\ar (z',k_2+1))+A((z',k_2)\ar (x_2,k_2))\\
=&T(x_1,k_2+1)+A((x_1,k_2+1)\ar  (x_2,k_2).
\end{align*}
This proves \eqref{equ:Action+LPP_1} for the case $k_1-k_2=1$. 

Assume $k_1\geq k_2+2$ and \eqref{equ:Action+LPP_1} holds with $k_1$ replaced by $k_1-1$ or $k_2$ replaced by $k_1-1$. Let $\gamma\in \mathcal{P}((x_1,k_1)\ar (x_2,k_2))$ be a geodesic. Set 
$$
z'=\min\{z\in [x_1,x_2]\, :\,  (z,k_1-1)\in\gamma\}.
$$
Then 
$$
A((x_1,k_1)\ar (x_2,k_2))=A((x_1,k_1)\ar (z',k_1))+A((z',k_1-1)\ar (x_2,k_2)).
$$
Therefore,
\begin{align*}
T(x_2,k_2)\geq &T(z',k_1-1)+A((z',k_1-1)\ar (x_2,k_2))\\
\geq& T(x_1,k_1)+A((x_1,k_1)\ar (z',k_1))+A((z',k_1-1)\ar (x_2,k_2))\\
=&T(x_1,k_1)+A((x_1,k_1)\ar  (x_2,k_2).
\end{align*}
\end{proof}

\begin{lemma}\label{lem:LPP_as_almostAction}
Let $A$ be a continuous environment on $M=I\times J$ and fix $q_0=(x_0,k_0)\in M$. Define $T$ on $M$ by
\begin{align*}
T(p)=\left\{ \begin{array}{cc}
A(q_0 \ar p), & \textup{if}\ q_0\preceq p,\\
-\infty, & \textup{otherwise}. 
\end{array} \right.
\end{align*}
Then for any $k\in J$ and $x,y\in I$ with $x<y$, the action recursion \eqref{equ:Local_metric_composition_law} holds provided $k<k_0$ or $x\geq x_0$.
\end{lemma}
\begin{proof}
Fix $x<y\in I$ and $k\in J$. We separate the discussion into four cases: $k>k_0$ and $x\geq x_0$; $k=k_0$ and $x\geq x_0$; $k<k_0$ and $y<x_0$; $k<k_0$ and $y\geq x_0$.

Suppose $k>k_0$ and $x\geq x_0$. Then both sides of \eqref{equ:Local_metric_composition_law} are equal to $-\infty$. Suppose $k=k_0$ and $x\geq x_0$. It is direct to check that $T(x,k)=A(x,k)-A(x_0,k)$, $T(y,k)=A(y,k)-A(x_0,k)$, and $T(z,k+1)\equiv -\infty$. Hence \eqref{equ:Local_metric_composition_law} holds. Suppose $k<k_0$ and $y<x_0$, it is easy to check both sides of \eqref{equ:Local_metric_composition_law} equal $-\infty$. Thus, \eqref{equ:Local_metric_composition_law} holds in these three cases.

It remains to prove \eqref{equ:Local_metric_composition_law} for the last case, $y\geq x_0$ and $k< k_0$, which is assumed for the rest of the proof. We aim to show that 
\begin{align}\label{equ:LPPAction1}
T(y,k) \geq (T(x,k)+A(y,k)-A(x,k) ) \vee \sup_{z\in [x,y]}\big( T(z,k+1)+A(y,k)-A(z,k)\big).
\end{align}
and
\begin{align}\label{equ:LPPAction2}
T(y,k) \leq (T(x,k)+A(y,k)-A(x,k) ) \vee \sup_{z\in [x,y]}\big( T(z,k+1)+A(y,k)-A(z,k)\big).
\end{align}

We begin with \eqref{equ:LPPAction1}. If $x<x_0$, then $T(x,k)=-\infty$ and $T(y,k) \geq T(x,k)+A(y,k)-A(x,k)$. If $x\geq x_0$, $T(y,k) \geq T(x,k)+A(y,k)-A(x,k)$ follows from the triangle inequality \eqref{equ:triangle}. In short, we have
\begin{equation}\label{equ:LPPAction3}
T(y,k) \geq T(x,k)+A(y,k)-A(x,k).
\end{equation}
From $T(z,k+1)\leq T(z,k)$, we can apply a similar argument to deduce
\begin{equation}\label{equ:LPPAction4}
T(y,k) \geq T(z,k+1)+A(y,k)-A(z,k).
\end{equation}
for all $z\in [x,y]$. Combining \eqref{equ:LPPAction3} and \eqref{equ:LPPAction4} yields \eqref{equ:LPPAction1}.

Next, we turn to \eqref{equ:LPPAction2}. Let $\gamma\in\mathcal{P}((x_0,k_0)\ar (y,k))$ be a geodesic. Set
$$
z'=\min\{z\in [x_0,y]\, :\, (z,k)\in\gamma \}.
$$ 
If $z'\leq x$, it is direct to check that $T(y,k)=T(x,k)+A(y,k)-A(x,k)$ and \eqref{equ:LPPAction2} holds. If $z'> x$, it is direct to check that $T(y,k)=T(z',k+1)+A(y,k)-A(z',k)$ and \eqref{equ:LPPAction2} holds. We note that in this case we used the assumption $k<k_0$. Combining the above yields \eqref{equ:LPPAction2} and completes the proof.
\end{proof}

\subsection{Connections between actions and geodesics}\label{sec:connection}

In this subsection, we discuss the connections between actions and geodesics. Our goal is to prove Proposition~\ref{prop:decomposition}, which expresses the action values in a variational form involving infinite geodesics. For the Airy line ensemnle, this  formula will correspond to the variational formula \eqref{e:ASvari} in the Airy sheet.

From now on, we specialize to the case where $A$ is continuous and is defined on $\mathbb{R}\times\mathbb{N}$. An \textbf{infinite directed path} is a directed path $\gamma$ that belongs to $\mathcal{P}((-\infty,\ast)\ar (x,k))$ for some $(x,k)\in\mathbb{R}\times\mathbb{N}$. For an infinite directed path $\gamma$, we write $q_\gamma=(z_\gamma,k_\gamma)$ to be the element in $\mathbb{R}\times\mathbb{N}$ such that $\gamma\in \mathcal{P}((-\infty,\ast)\ar q_\gamma)$. We call $q_\gamma$ the end point of $\gamma$. We set $a_\gamma=\inf\{x:(x,k)\in \gamma\ \textup{for some}\ k\in\mathbb{N}\}$. An {\bf infinite geodesic} is an infinite directed path which is a geodesic.

Given an infinite geodesic $\gamma$, we can define an action associated to $\gamma$. For any $q\in\gamma\cap(\mathbb{R}\times\mathbb{N})$, we define $T_{q,\gamma}$ on $\mathbb{R}\times\mathbb{N}$ by
\begin{align*}
T_{q,\gamma}(p):=\left\{ \begin{array}{cc}
A(q\ar p)-A(q\ar q_\gamma), & \textup{if}\ q\preceq p ,\\
-\infty, & \textup{otherwise}. 
\end{array} \right.
\end{align*}
\begin{lemma}\label{lem:Action_from_geodesic}
Let $A$ be a continuous environment on $\mathbb{R}\times\mathbb{N}$ and $\gamma$ be an infinite geodesic with respect to $A$. Then $T_{q,\gamma}$ is nonincreasing in $q$ with respect to the ordering $\preceq$. 
\end{lemma}
\begin{proof}
Fix $p\in \mathbb{R}\times \mathbb{N}$ and $q_1\preceq q_2$ in $\gamma\cap(\mathbb{R}\times\mathbb{N})$. We aim to show $T_{q_1,\gamma}(p)\geq T_{q_2,\gamma}(p)$. This is obvious when $T_{z_2,\gamma}(p)=-\infty$. Hence we may assume $q_2\preceq p$. From the definition of geodesics, we have 
$$A(q_2\ar q_\gamma)=-A(q_1\ar q_2)+A(q_1\ar q_\gamma).$$
Therefore, 
\begin{align*}
T_{q_2,\gamma}(p)=A(q_2\ar p)-A(q_2\ar  q_\gamma)=&A(q_2\ar p)+A(q_1\ar q_2)-A(q_1\ar q_\gamma)\\
\leq &A(q_1\ar p)-A(q_1\ar q_\gamma)=T_{q_1,\gamma}(p).
\end{align*}
The triangle inequality \eqref{equ:triangle} was used in the above inequality. This completes the proof. 
\end{proof}
\begin{definition}\label{def:Action_from_geodesic}
Let $A$ be a continuous environment  on $\mathbb{R}\times\mathbb{N}$ and $\gamma$ be an infinite geodesic with respect to $A$. Let $\{q_j\}_{j=1}^\infty$ be a sequence in $\gamma\cap (\mathbb{R}\times\mathbb{N})$ such that
\begin{itemize}
\item $q_j$ is nonincreasing in $j$ with respect to the ordering $\preceq$;
\item for all $q\in \gamma\cap (\mathbb{R}\times\mathbb{N})$, $q_j\preceq q$ for $j$ large enough.
\end{itemize}
We define $T_\gamma$ on $\mathbb{R}\times\mathbb{N}$ by
\begin{equation}\label{equ:T_gamma_1}
T_{\gamma}(p) = \lim_{j\to\infty} T_{q_j,\gamma}(p).
\end{equation}
We note that from Lemma~\ref{lem:Action_from_geodesic}, the above limit exists (could be $\pm\infty$) and does not depend on the sequence $\{q_j\}_{j=1}^\infty$.
\end{definition}

\begin{lemma}\label{lem:Bgamma_action}
$T_\gamma$ is an action.  
\end{lemma}
\begin{proof}
We need to show that $T_\gamma$ satisfies the action recursion \eqref{equ:Local_metric_composition_law}. Fix $x<y$ and $k\in\mathbb{N}$. Let $q_j=(x_j,k_j)$ be a sequence as described in Definition~\ref{def:Action_from_geodesic}. For $j$ large enough, we have $x\geq x_j$ or $k<k_j$. From Lemma~\ref{lem:LPP_as_almostAction}, we have
\begin{align*} 
T_{q_j,\gamma}(y,k) =(T_{q_j,\gamma}(x,k)+A(y,k)-A(x,k) ) \vee \sup_{z\in [x,y]}\big( T_{q_j,\gamma}(z,k+1)+A(y,k)-A(z,k)\big).
\end{align*}
Since $T_{q_j,\gamma}$ is nondecreasing in $j$, the assertion holds upon taking $j\to\infty$.
\end{proof}

Let $\gamma$ be an infinite geodesic. It is direct to check that   for $q_1\preceq q_2$ in $\gamma\cap(\mathbb{R}\times\mathbb{N})$,
\begin{align*}
T_\gamma(q_2)=T_\gamma(q_1)+A(q_1\ar q_2).
\end{align*}

\begin{definition}
Let $T$ be an action. A geodesic $\gamma$ is called a $T${\bf -geodesic} if 
\begin{align}\label{equ:T_geodesic}
T(q_2)=T(q_1)+A(q_1\ar q_2)
\end{align}
for all $q_1\preceq q_2$ in $\gamma\cap(\mathbb{R}\times\mathbb{N})$. An infinite $T${\bf -geodesic} is a $T${\bf -geodesic} which is infinite. An {\bf infinite $(T,q)$-geodesic} is defined as an infinite $T$-geodesic with end point $q$. 
\end{definition}

The preceding discussion shows that an infinite geodesic $\gamma$ is always an infinite $(T_\gamma,q_\gamma)$-geodesic. We note that this relationship is not a one-to-one correspondence; for an infinite $(T,q)$-geodesic $\gamma$, $T$ is typically not equal to $T_\gamma$. 

\begin{lemma}\label{l:decomp}
Let $T$ be an action, $q$ belong to $\mathbb{R}\times\mathbb{N}$ with $T(q)\in\mathbb{R}$, and $\gamma$ be an infinite $(T,q)$-geodesic. Then $T_\gamma(p)\le T(p)-T(q)$ for all $p\in\mathbb{R}\times\mathbb{N}$. 
\end{lemma}
\begin{proof}
Since $q$ is the end point of $\gamma$, we have from \eqref{equ:T_geodesic} that
\begin{align}\label{equ:decomp_1}
T(q')=T(q)-A(q'\ar q)
\end{align}
for all $q'\in\gamma\cap (\mathbb{R}\times\mathbb{N})$. In particular, $T(q')\in\mathbb{R}$. 

Fix $p\in\mathbb{R}\times\mathbb{N}$ and let $q_j$ be a sequence as described in Definition~\ref{def:Action_from_geodesic}. The assertion is obvious if $T_\gamma(p)=-\infty$. Hence we may assume $q_j\preceq p$ for $j$ large enough. Then $T_{q_j,\gamma}(p)=A(q_j\ar p)-A(q_j\ar q )$. From \eqref{equ:decomp_1}, we have 
$$
T_{q_j,\gamma}(p)=A(q_j\ar p)-T(q)+T(q_j).
$$ Applying \eqref{equ:Action+LPP}, we conclude
\begin{align*}
T_{q_j,\gamma}(p) \leq T(p)-T(q).
\end{align*}
Taking $j\to\infty$ yields the desired assertion.
\end{proof}

\begin{proposition}\label{prop:Bq_geodesic}
Let $T$ be an upper semicontinuous action on $\mathbb{R}\times\mathbb{N}$. Then for any $q\in\mathbb{R}\times\mathbb{N}$ with $T(q)\in\mathbb{R}$, there exists an infinite $(T,q)$-geodesic.
\end{proposition}

We postpone the proof of Proposition~\ref{prop:Bq_geodesic} to the end of this section. Next, we use Proposition~\ref{prop:Bq_geodesic} to prove Proposition~\ref{prop:decomposition}, which will play an important role in Section~\ref{sec:Airyaction}.

\begin{proposition}\label{prop:decomposition}
Let $T$ be an upper semicontinuous action on $\mathbb{R}\times\mathbb{N}$. Define
$$\Omega_T=\{ x\in\mathbb{R}\, :\, T(x,1)\in\mathbb{R} \}.$$
For each $x\in \Omega_T$, let $\gamma(x)$ be an infinite $(T,(x,1))$-geodesic given by Proposition~\ref{prop:Bq_geodesic}. Then for all $y\in \Omega_T$, we have
\begin{equation}\label{e:decomp}
T(y,1)=\sup\{ T_{\gamma(x)}(y,1)+T(x,1)\, :\, x\in\Omega_T\}.
\end{equation}
\end{proposition}
\begin{proof}
Fix $y\in\Omega_T$. From Lemma~\ref{l:decomp}, 
$$
T(1,y)\geq  T_{\gamma(x)}(1,y)+T(1,x) 
$$
for all $x\in\Omega_T$. Hence 
\begin{equation*} 
T(1,y)\geq \sup\{ T_{\gamma(x)}(1,y)+T(1,x)\, :\, x\in\Omega_T\}.
\end{equation*}
From $T_{\gamma(y)}(1,y)=0$, we have $T(1,y)=T_{\gamma(y)}(1,y)+T(1,y).$
Hence
$$
T(1,y)\leq \sup\{ T_{\gamma(x)}(1,y)+T(1,x)\, :\, x\in\Omega_T\}.
$$
Combining the above inequalities yields \eqref{e:decomp}.
\end{proof}

The rest of the section is devoted to proving Proposition~\ref{prop:Bq_geodesic}. We need some preliminary lemmas.

\begin{lemma}\label{lem:geodesic_T}
Let $T$ be an action on $\mathbb{R}\times\mathbb{N}$ and $\gamma$ be a directed path. Assume that 
\begin{align}\label{equ:geodesic_T}
T(q_2)=T(q_1)+A(q_1\ar q_2)
\end{align}
for all $q_1\preceq q_2$ in $\gamma\cap(\mathbb{R}\times\mathbb{N})$ and that $T(q_0)\in\mathbb{R}$ for some $q_0\in \gamma\cap(\mathbb{R}\times\mathbb{N}) $. Then $\gamma$ is a $T$-geodesic. Moreover, $T(q)\in \mathbb{R}$ for all $q\in \gamma\cap(\mathbb{R}\times\mathbb{N})$.
\end{lemma}
\begin{proof}
Note from \eqref{equ:geodesic_T} and the assumption, we have $T(q)\in\mathbb{R}$ for all $q\in\gamma\cap(\mathbb{R}\times\mathbb{N})$. Now we show $\gamma$ is a $T$-geodesic. Fix $q_1\preceq q_2\preceq q_3\in \gamma\cap(\mathbb{R}\times\mathbb{N})$. From the assumption we have
\begin{align*}
A(q_1\ar q_2)+A(q_2\ar q_3)=&T(q_2)-T(q_1)+T(q_3)-T(q_2)\\
=&T(q_3)-T(q_1)=A(q_1\ar q_3).
\end{align*}
This shows that $\gamma$ is a geodesic. Since the condition~\eqref{equ:T_geodesic} is assumed in \eqref{equ:geodesic_T}, the proof is complete.
\end{proof}
\begin{lemma}\label{lem:glue}
Let $q'\in \mathbb{R}\times\mathbb{N}$ and $\gamma_1$, $\gamma_2$ be $T$-geodesics. Assume that $q\in\gamma_1\cap\gamma_2$, $q\preceq q_2$ for all $q_2\in\gamma_2$ and $q_1\preceq q$ for all $q_1\in\gamma_1$. Further assume there exists $q_0\in\gamma_2$ such that $T(q_0)\in\mathbb{R}$. Then $\gamma=\gamma_1\cup\gamma_2$ is a $T$-geodesic.
\end{lemma}
\begin{proof}
It is direct to check that $\gamma=\gamma_1\cup\gamma_2$ is a directed path. Applying Lemma~\ref{lem:geodesic_T} to $\gamma_2$ and then to $\gamma_1$, we deduce that $T(q)\in\mathbb{R}$ for all $q\in\gamma\cap (\mathbb{R}\times\mathbb{N})$. To prove $\gamma$ is a $T$-geodesic, in view of Lemma~\ref{lem:geodesic_T}, it suffices to verify \eqref{equ:geodesic_T}. Fix $q_1,q_2\in \gamma\cap (\mathbb{R}\times\mathbb{N})$ with $q_1\preceq q_2$. Suppose $q_1,q_2\in \gamma_1\cap (\mathbb{R}\times\mathbb{N})$ or $q_1,q_2\in \gamma_2\cap (\mathbb{R}\times\mathbb{N})$, then \eqref{equ:geodesic_T} holds since each $\gamma_1$ and $\gamma_2$ are $T$-geodesics. Suppose $q_1\in\gamma_1$ and $q_2\in\gamma_2$. From \eqref{equ:Action+LPP}, we have $T(q_2)-T(q_1)\geq A(q_1\ar q_2)$. Applying the triangle inequality \eqref{equ:triangle} and the fact that $\gamma_i$ are $T$-geodesics,
\begin{align*}
A(q_1\ar q_2)\geq A(q_1\ar q)+A(q\ar q_2)=T(q)-T(q_1)+T(q_2)-T(q)=T(q_2)-T(q_1).
\end{align*}
This concludes \eqref{equ:geodesic_T} and completes the proof.
\end{proof}

Let $T$ be an upper semicontinuous action on $\mathbb{R}\times\mathbb{N}$. We associate $T$ with a map $F_T$ from $\mathbb{R}\times\mathbb{N}$ to $\mathbb{R}\cup\{-\infty\}$ as follows. Fix $(x,k)\in \mathbb{R}\times\mathbb{N}$. Consider the set
\begin{align*}
G_T(x,k)=\{y\in (-\infty,x] \, :\, T(y,k+1)+A(x,k)-A(y,k)=T(x,k)\}.
\end{align*}
From \eqref{equ:Local_metric_composition_law}, this set is equal to
\begin{align*}
G_T(x,k)=\{y\in (-\infty,x] \, :\, T(y,k+1)+A(x,k)-A(y,k)\geq T(x,k)\}.
\end{align*}
Since $T$ is upper semicontinuous, $G_T(x,k)$ is a closed set. We define $F_T(x,k)=\sup G_T(x,k)$. Suppose $G_T(x,k)$ is empty and $F_T(x,k)=-\infty$, we set $\gamma(x,k)=(-\infty,x]\times \{k\}$. Suppose $G_T(x,k)$ is non-empty. We denote $F_T(x,k)=x'\in (-\infty,x]$ and set $\gamma(x,k)=[x',x]\times \{k\}\cup \{(x',k+1)\}$. 
\begin{lemma}
Let $T$ be an upper semicontinuous action on $\mathbb{R}\times\mathbb{N}$. Fix $(x,k)\in \mathbb{R}\times\mathbb{N}$. Suppose $F_T(x,k)=x'\in (-\infty,x]$. Then 
\begin{equation}\label{equ:segment_1}
T(y,k)=T(x,k)-A(x,k)+A(y,k)\ \textup{for all}\ y\in [x',x].
\end{equation}
Moreover,
\begin{equation}\label{equ:jump_x'}
T(x',k+1)=T(x',k).
\end{equation}

Suppose $F_T(x,k)=-\infty$. Then
\begin{equation}\label{equ:segment_2}
T(y,k)=T(x,k)-A(x,k)+A(y,k)\ \textup{for all}\ y\in (-\infty,x].
\end{equation}
\end{lemma}
\begin{proof}
Suppose $F_T(x,k)=x'\in (-\infty,x]$. Then $T(z,k+1)+A(x,k)-A(z,k)<T(x,k)$ for all $z\in (x',x]$. From the action recursion \eqref{equ:LPPAction1} and the upper semicontinuity of $T$, we deduce \eqref{equ:segment_1} for all $y\in (x',x]$. Using the definition of $F_T(x,k)=x'$ and applying \eqref{equ:Local_metric_composition_law} twice, we deduce
\begin{align*}
T(x,k)=T(x',k+1)+A(x,k)-A(x',k)\leq T(x',k)+A(x,k)-A(x',k) \leq T(x,k).
\end{align*}
Hence \eqref{equ:jump_x'} holds and \eqref{equ:segment_1} holds for $y=x'$. The proof for \eqref{equ:segment_2} is similar and we omit the details.
\end{proof}

\begin{lemma}\label{lem:one_piece}
For any $(x,k)\in \mathbb{R}\times\mathbb{N}$ with $T(x,k)\in\mathbb{R}$, $\gamma(x,k)$ is a $T$-geodesic.
\end{lemma}
\begin{proof}
It is direct to check that $\gamma(x,k)$ is a directed path. In view of Lemma~\ref{lem:geodesic_T}, it suffices to check the condition~\eqref{equ:geodesic_T}. 

We begin with the case $F_T(x,k)=-\infty$ and $\gamma(x,k)=(-\infty,x]\times \{k\}$. From \eqref{equ:segment_2}, we have
\begin{equation*}
T(y,k)=T(x,k)-A(x,k)+A(y,k)\ \textup{for all}\ y\in (-\infty,x]. 
\end{equation*}
It is then direct to verify \eqref{equ:geodesic_T} for $\gamma(x,k)$. 

Next, we consider the case $F_T(x,k)=x'\in\mathbb{R}$ and $\gamma(x,k)=[x',x]\times \{k\}\cup \{(x',k+1)\}$. Set $\gamma_1(x,k)=\{x'\}\times \{k,k+1\}$ and $\gamma_2(x,k)=[x',x]\times \{k\}$. From \eqref{equ:jump_x'}, $T(x',k+1)=T(x',k)$. Hence \eqref{equ:geodesic_T} holds on $\gamma_1(x,k)=\{x'\}\times \{k,k+1\}$. From \eqref{equ:segment_1}, 
\begin{align*}
T(y,k)=T(x,k)-A(x,k)+A(y,k)\ \textup{for all}\ y\in [x',x].
\end{align*}
Hence \eqref{equ:geodesic_T} holds on $\gamma_2(x,k)=[x',x]\times \{k\}$. Applying Lemma \ref{lem:glue} shows $\gamma(x,k)=\gamma_1(x,k)\cup\gamma_2(x,k)$ is a $T$-geodesic. The proof is complete.
\end{proof}

\begin{proof}[Proof of Proposition~\ref{prop:Bq_geodesic}]
Define $\bar{F}_T:(\mathbb{R}\cup\{-\infty\})\times\mathbb{N}\to (\mathbb{R}\cup\{-\infty\})\times\mathbb{N}$ by 
\begin{align*}
\bar{F}_T(x,k)=\left\{ \begin{array}{cc}
(F_T(x,k),k+1), & x\in\mathbb{R},\\
(-\infty, k+1), & x=-\infty.
\end{array} \right.
\end{align*}
Fix $q=(x_0,k_0)\in\mathbb{R}\times\mathbb{N}$. We inductively define an nonincreasing sequence $\{x_j\}_{j=1}^\infty$ in $\mathbb{R}\cup\{-\infty\}$ by $(x_j,k_0+j)=\bar{F}_T(x_{j-1},k_0+j-1)$. We separate the discussion into two cases.

Assume $x_j>-\infty$ for all $j\geq 1$. Let $a_\gamma=\lim_{j\to\infty} x_j$. We define $\gamma$ by
\begin{align*}
\gamma=\cup_{j=0}^\infty \gamma(x_j,k_0+j)\cup (-\infty,a_\gamma]\times\{\infty\}.
\end{align*} 
Applying Lemmas~\ref{lem:glue} and \ref{lem:one_piece}, we deduce that $\gamma$ is an infinite $(T,q)$-geodesic.

Next, we assume there exists $j_0\geq 0$ such that $x_{j_0+1}=-\infty$ and $x_{j_0}>-\infty$ for $0\leq j\leq j_0$. We define $\gamma$ by
\begin{align*}
\gamma=\cup_{j=0}^{j_0} \gamma(x_j,k_0+j).
\end{align*}
Applying Lemmas~\ref{lem:glue} and \ref{lem:one_piece}, we deduce that $\gamma$ is an infinite $(T,q)$-geodesic.
\end{proof}

\section{Convergence of actions}

This section's primary results are Propositions~\ref{p:lp-to-lp}, \ref{p:polymer_to_lp}, \ref{p:polymer-to-polymer}, and \ref{p:polymer_to_lp_discrete}. These propositions establish the framework for the convergence of environments with associated actions. As the proofs are standard $\eps$-$\delta$ arguments, the reader may wish to read the statement of Proposition \ref{p:lp-to-lp} and to proceed to the next section, which is more conceptual and of greater interest. 

In this section we will use upper semicontinuous functions and the corresponding convergence notions (hypograph topology). See Appendix \ref{a:UC} for a review and notations.

\subsection{Last Passage to Last Passage}\label{sec:lp_to_lp}

\begin{proposition}\label{p:lp-to-lp}
Let $ \alpha_n\uparrow\infty$ be a sequence in $[0,\infty]$, $j_n\uparrow\infty$ be a sequence in $\mathbb{N}\cup\{\infty\}$.  Let $M_n=[-\alpha_n,\infty)\times \{k\in\mathbb{N}\ :\ k<j_n\}$. For each $n$, let $A_n:M_n\to\mathbb{R}$ be an environment and let $T_n:M_n\to\mathbb{R}\cup\{\pm\infty\}$ be an $A_n$-action which is upper semicontinuous. 

Suppose that there exists a continuous environment $A:\mathbb{R}\times\mathbb{N}\to\mathbb{R}$ such that $A_n$ converges to $A$ locally uniformly. Further assume that $T_n$ converges to $T$ in $\uc(\mathbb{R}\times\mathbb{N})$. Then $T$ is an $A$-action.   
\end{proposition}
\begin{proof}
We need to verify the action recursion \eqref{equ:Local_metric_composition_law}. Fix $x<y$ and $k\in\mathbb{N}$. We aim to show
\begin{equation}\label{equ:lp-to-lp_1}
T(y,k) \leq (T(x,k)+A(y,k)-A(x,k) ) \vee \sup_{z\in [x,y]}\big( T(z,k+1)+A(y,k)-A(z,k)\big),
\end{equation} 
\begin{equation}\label{equ:lp-to-lp_2}
T(y,k)\geq T(x,k)+A(y,k)-A(x,k),
\end{equation}
and
\begin{equation}\label{equ:lp-to-lp_3}
T(y,k)\geq \sup_{z\in [x,y]}\big( T(z,k+1)+A(y,k)-A(z,k)\big).
\end{equation}
We begin with \eqref{equ:lp-to-lp_1}. For any $0<\varepsilon<1$, let $0<\delta<1\wedge (y-x)$ be a positive number such that $T(z,k+1)\leq T(y,k+1)+\varepsilon$ for all $z\in [y,y+\delta]$ and $|A(w,k)-A(y,k)|\leq \varepsilon$ for all $|w-y|\leq \delta$. Let $n$ be large enough such that $|A_n(w,k)-A(w,k)|\leq \varepsilon$ for all $w\in [x-1,y+1]$. Fix arbitrary $y'\in (y-\delta,y+\delta)$. From \eqref{equ:Local_metric_composition_law}, we have
\begin{align*}
T_n(y',k) \leq (T_n(x,k)+A_n(y',k)-A_n(x,k) ) \vee \sup_{z\in [x,y']}\big( T_n(z,k+1)+A_n(y',k)-A_n(z,k)\big).
\end{align*}
From the above arrangement,
\begin{align*}
T_n(x,k)+A_n(y',k)-A_n(x,k)\leq & T_n(x,k)+A(y',k)-A(x,k)+2\varepsilon\\
\leq &T_n(x,k)+A(y,k)-A(x,k)+3\varepsilon.
\end{align*}
Similarly,  
\begin{align*}
\sup_{z\in [x,y']}\big( T_n(z,k+1)+A_n(y',k)-A_n(z,k)\big)\leq \sup_{z\in [x,y+\delta]}\big( T_n(z,k+1)+A(y,k)-A(z,k)\big)+3\varepsilon.
\end{align*}
Therefore,
\begin{align*}
\sup_{y'\in (y-\delta,y+\delta)} &T_n(y',k) \\ &\leq (T_n(x,k)+A (y ,k)-A (x,k) ) \vee \!\!\! \sup_{z\in [x,y+\delta]}\big( T_n(z,k+1)+A (y ,k)-A (z,k)\big)+3\varepsilon.
\end{align*}
Taking $\limsup_{n\to\infty}$ and applying \eqref{equ:UC_limit_closed} and \eqref{equ:UC_limit_open}, we get
\begin{align*}
\sup_{y'\in (y-\delta,y+\delta)} &T(y',k) \\&\leq (T(x,k)+A (y ,k)-A (x,k) ) \vee \sup_{z\in [x,y+\delta]}\big( T(z,k+1)+A (y ,k)-A (z,k)\big)+3\varepsilon.
\end{align*}
For $z\in [y,y+\delta]$, we have
\begin{align*}
T(z,k+1)+A (y ,k)-A (z,k)\leq T(y,k+1)+2\varepsilon.
\end{align*}
Therefore,
\begin{align*}
T(y,k) \leq (T(x,k)+A(y,k)-A(x,k) ) \vee \sup_{z\in [x,y]}\big( T(z,k+1)+A(y,k)-A(z,k)\big)+5\varepsilon.
\end{align*}
Sending $\varepsilon\to 0$ yields \eqref{equ:lp-to-lp_1}.

Next, we turn to \eqref{equ:lp-to-lp_2}. Fix any $0<\varepsilon<1$. Let $0<\delta<1\wedge (y-x)$ be a positive number such that $|A(x',k)-A(x,k)|\leq \varepsilon$ for all $|x'-x|\leq \delta$. Let $n$ be large enough such that $|A_n(w,k)-A(w,k)|\leq \varepsilon$ for all $w\in [x-1,y+1]$. Fix arbitrary $x'\in (x-\delta,x+\delta)$. From \eqref{equ:Local_metric_composition_law},
\begin{align*}
T_n(y,k)\geq T_n(x',k)+A_n(y,k)-A_n(x',k)\geq T_n(x',k)+A(y,k)-A(x',k)-2\varepsilon\\
\geq T_n(x',k)+A(y,k)-A(x,k)-3\varepsilon.
\end{align*}
Hence
\begin{align*}
T_n(y,k)\geq \sup_{x'\in (x-\delta,x+\delta)}T_n(x',k)+A(y,k)-A(x,k)-3\varepsilon.
\end{align*}
Taking $\limsup_{n\to\infty}$ and applying \eqref{equ:UC_limit_closed} and \eqref{equ:UC_limit_open}, we get
\begin{align*}
T(y,k)\geq T(x,k)+A(y,k)-A(x,k)-3\varepsilon.
\end{align*}
Sending $\varepsilon\to 0$ yields \eqref{equ:lp-to-lp_2}.

Lastly, we prove \eqref{equ:lp-to-lp_3}. An argument similar to the above one shows for any $w'>w$, $T(w',k)\geq  T(w,k+1)+A(w',k)-A(w,k)$. Taking $w'=y$ and $w=z$ gives
\begin{align*}
T(y,k)\geq  \sup_{z\in [x,y)}\big( T(z,k+1)+A(y,k)-A(z,k)\big).
\end{align*}
Taking $w=y$ and $w'\downarrow y$ gives
\begin{align*}
T(y,k)\geq T(y,k+1).
\end{align*}
Combining the above yields \eqref{equ:lp-to-lp_3}. The proof is complete.
\end{proof}

\subsection{Direct Polymer to Last Passage}

In semi-discrete polymer models, actions can be defined to match how the point-to-point free energy evolves.

\begin{definition}
Let $A$ be an environment defined on $M=I\times J$ and $w>1$. A function $T:M\to\mathbb{R}\cup\{\pm\infty\}$ is called a $w$\textbf{-polymer} $A$\textbf{-action} if for all $x<y$ in $I$ and $k\in J$, we have
\begin{equation}\label{e:polymer-mc}
w^{T(y,k)} = w^{ T(x,k)+A(y,k)-A(x,k)}+\int_x^y w^{T(z,k+1)+A(y,k)-A(z,k)}dz.
\end{equation}
Here we adopt the convention that $T(z,k+1)\equiv -\infty$ if $k+1\notin J$, $w^\infty=\infty$ and $w^{-\infty}=0$.  We call \eqref{e:polymer-mc} the $w$-polymer action recursion. In words, $w^T$ satisfies the integral version of the semi-discrete heat equation. In the literature, $w$ is typically replaced by $e^\beta$ where $\beta>0$ is the inverse temperature. For simplicity, we will use $w$ throughout this work.
\end{definition}

Taking $w\to\infty$ in \eqref{e:polymer-mc} formally yields the action recursion \eqref{equ:Local_metric_composition_law}. Below we show that \eqref{e:polymer-mc} implies an approximate version of \eqref{equ:Local_metric_composition_law}. This lemma is a quantitative version of the Euler-Lagrange transition as $w\to\infty$. 

\begin{lemma}\label{l:EL}
Let $A$ be an environment defined on $M=I\times J$, $w>1$ and let $T$ be a $w$-polymer $A$-action. Then for all $x<y$ in $I$ and $k\in J$, we have
\begin{equation}\label{e:EL-le}
T(y,k)\le (T(x,k)+A(y,k)-A(z,k))\vee
\sup_{z\in [x,y]}(T(z,k+1)+A(y,k)-A(z,k))+\log_w(1+y-x).
\end{equation}
\begin{equation}\label{e:EL-ge_1}
T(y,k)\ge T(x,k)+A(y,k)-A(z,k).
\end{equation}
Let $\delta\in(0,y-x)$, then
\begin{equation}\label{e:EL-ge_2}
T(y,k)\ge \sup_{z\in [x,y-\delta]}(T(z,k+1)+A(y,k)-A(z,k))+\log_w\delta -m(\delta)
\end{equation}
with the one-sided modulus of continuity
$$m(\delta)
=\sup_{\stackrel{z,z'\in[x,y]}{0<z'-z\le \delta}}A(z',k)-A(z,k)-A(z',k+1)+A(z,k+1).
$$
\end{lemma}
\begin{proof}
Denote
$$Z=(T(x,k)+A(y,k)-A(z,k))\vee
\sup_{z\in [x,y]}(T(z,k+1)+A(y,k)-A(z,k)).$$
Then from \eqref{e:polymer-mc},
\begin{align*}
w^{T(y,k)} = &w^{ T(x,k)+A(y,k)-A(x,k)}+\int_x^y w^{T(z,k+1)+A(y,k)-A(z,k)}dz\\
\leq & w^Z+(y-x)w^Z=(y-x+1)w^Z.
\end{align*}
Taking the logarithm of the above gives \eqref{e:EL-le}.

The bound \eqref{e:EL-ge_1} direct follows by dropping the integral term from \eqref{e:polymer-mc}. Now we turn to \eqref{e:EL-ge_2}. From \eqref{e:EL-ge_1}, we have for all $z\leq z'$ in $I$ 
$$
T(z',k+1)-T(z,k+1)\ge A(z',k+1)-A(z,k+1).
$$
Therefore, for $z\in [x,y-\delta]$ and $z'\in [z,z+\delta]$,
\begin{align*}
T(z',k+1)+A(y,k)-A(z',k)\geq& \big(T(z,k+1)+A(y,k)-A(z,k)\big) \\
&-\big(A(z',k)-A(z,k)-A(z',k+1)+A(z,k+1)\big)\\
\geq &T(z,k+1)+A(y,k)-A(z,k)-m(\delta).
\end{align*}
Hence
\begin{align*}
\int_x^y w^{T(z',k+1)+A(y,k)-A(z',k)}\,dz'&\ge 
 \sup_{z\in [x,y-\delta]} \int_z^{z+\delta} w^{T(z',k+1)+A(y,k)-A(z',k)}\,dz'\\
 &\ge \delta \sup_{z\in [x,y-\delta]} w^{T(z,k+1)+A(y,k)-A(z,k)-m(\delta)}.
\end{align*}
Taking the logarithm of the above gives \eqref{e:EL-ge_2}.
\end{proof}
\begin{proposition}\label{p:polymer_to_lp}
Let $ \alpha_n\uparrow\infty$ be a sequence in $[0,\infty]$, $j_n\uparrow\infty$ be a sequence in $\mathbb{N}\cup\{\infty\}$, and $w_n\uparrow\infty$ be a sequence in $(1,\infty)$. Let $M_n=[-\alpha_n,\infty)\times \{k\in\mathbb{N}\ :\ k<j_n\}$. For each $n$, let $A_n:M_n\to\mathbb{R}$ be an environment and let $T_n:M_n\to\mathbb{R}\cup\{\pm\infty\}$ be a $w_n$-polymer $A_n$-action which is upper semicontinuous. 

Suppose that there exists a continuous environment $A:\mathbb{R}\times\mathbb{N}\to\mathbb{R}$ such that $A_n$ converges to $A$ locally uniformly. Further assume $T_n$ converges to $T$ in $\uc(\mathbb{R}\times\mathbb{N})$. Then $T$ is an $A$-action.   
\end{proposition}
The proof of this result builds upon the proof of Proposition \ref{p:lp-to-lp}, and we recommend that the reader first read that proof.
\begin{proof}
We need to verify the action recursion \eqref{equ:Local_metric_composition_law}. Fix $x<y$ and $k\in\mathbb{N}$. We aim to show
\begin{equation}\label{equ:polymer-to-lp_1}
T(y,k) \leq (T(x,k)+A(y,k)-A(x,k) ) \vee \sup_{z\in [x,y]}\big( T(z,k+1)+A(y,k)-A(z,k)\big),
\end{equation} 
\begin{equation}\label{equ:polymer-to-lp_2}
T(y,k)\geq T(x,k)+A(y,k)-A(x,k),
\end{equation}
and
\begin{equation}\label{equ:polymer-to-lp_3}
T(y,k)\geq \sup_{z\in [x,y]}\big( T(z,k+1)+A(y,k)-A(z,k)\big)
\end{equation}
We begin with the proof of \eqref{equ:polymer-to-lp_1}. For any $0<\varepsilon<1$, let $0<\delta<1\wedge (y-x)$ be a positive number such that $T(z,k+1)\leq T(y,k+1)+\varepsilon$ for all $z\in [y,y+\delta]$ and $|A(w,k)-A(y,k)|\leq \varepsilon$ for all $|w-y|\leq \delta$. Let $n$ be large enough such that $|A_n(w,k)-A(w,k)|\leq \varepsilon$ for all $w\in [x-1,y+1]$. Fix arbitrary $y'\in (y-\delta,y+\delta)$. From \eqref{e:EL-le}, we have
\begin{align*}
T_n(y',k) \leq (T_n(x,k)+A_n(y',k)-A_n(x,k) ) \vee \sup_{z\in [x,y']}\big( T_n(z,k+1)+A_n(y',k)-A_n(z,k)\big)\\
+\log_{w_n}(2+y-x).
\end{align*}
From the above arrangement,
\begin{align*}
T_n(x,k)+A_n(y',k)-A_n(x,k)\leq & T_n(x,k)+A(y',k)-A(x,k)+2\varepsilon\\
\leq &T_n(x,k)+A(y,k)-A(x,k)+3\varepsilon.
\end{align*}
Similarly, 
\begin{align*}
\sup_{z\in [x,y']}\big( T_n(z,k+1)+A_n(y',k)-A_n(z,k)\big)\leq \sup_{z\in [x,y+\delta]}\big( T_n(z,k+1)+A(y,k)-A(z,k)\big)+3\varepsilon.
\end{align*}
Therefore,
\begin{align*}
\sup_{y'\in (y-\delta,y+\delta)} &T_n(y',k) \leq (T_n(x,k)+A (y ,k)-A (x,k) ) \\&\vee \sup_{z\in [x,y+\delta]}\big( T_n(z,k+1)+A (y ,k)-A (z,k)\big)
+\log_{w_n}(2+y-x)+3\varepsilon.
\end{align*}
Taking $\limsup_{n\to\infty}$ and applying \eqref{equ:UC_limit_closed} and \eqref{equ:UC_limit_open}, we get
\begin{align*}
\sup_{y'\in (y-\delta,y+\delta)} T(y',k)\leq (T(x,k)+A (y ,k)\!-\! (x,k) ) \vee \!\!\! \sup_{z\in [x,y+\delta]}\big( T(z,k+1)+A (y ,k)\!-\!A (z,k)\big) +3\varepsilon.
\end{align*}
For $z\in [y,y+\delta]$, we have
\begin{align*}
T(z,k+1)+A (y ,k)-A (z,k)\leq T(y,k+1)+2\varepsilon.
\end{align*}
Therefore,
\begin{align*}
T(y,k) \leq (T(x,k)+A(y,k)-A(x,k) ) \vee \sup_{z\in [x,y]}\big( T(z,k+1)+A(y,k)-A(z,k)\big)+5\varepsilon.
\end{align*}
Sending $\varepsilon\to 0$ yields \eqref{equ:polymer-to-lp_1}.

Next, we turn to \eqref{equ:polymer-to-lp_2}. Fix any $0<\varepsilon<1$. Let $0<\delta<1\wedge (y-x)$ be a positive number such that $|A(x',k)-A(x,k)|\leq \varepsilon$ for all $|x'-x|\leq \delta$. Let $n$ be large enough such that $|A_n(w,k)-A(w,k)|\leq \varepsilon$ for all $w\in [x-1,y+1]$. Fix arbitrary $x'\in (x-\delta,x+\delta)$. From \eqref{e:EL-ge_1},
\begin{align*}
T_n(y,k)\geq T_n(x',k)+A_n(y,k)-A_n(x',k)\geq T_n(x',k)+A(y,k)-A(x',k)-2\varepsilon\\
\geq T_n(x',k)+A(y,k)-A(x,k)-3\varepsilon.
\end{align*}
Hence
\begin{align*}
T_n(y,k)\geq \sup_{x'\in (x-\delta,x+\delta)}T_n(x',k)+A(y,k)-A(x,k)-3\varepsilon.
\end{align*}
Taking $\limsup_{n\to\infty}$ and applying \eqref{equ:UC_limit_closed} and \eqref{equ:UC_limit_open}, we get
\begin{align*}
T(y,k)\geq T(x,k)+A(y,k)-A(x,k)-3\varepsilon.
\end{align*}
Sending $\varepsilon\to 0$ yields \eqref{equ:polymer-to-lp_2}.

Lastly, we prove \eqref{equ:polymer-to-lp_3}. An argument similar to the above one, with \eqref{e:EL-ge_1} replaced by \eqref{e:EL-ge_2}, shows for any $w'>w$, $T(w',k)\geq  T(w,k+1)+A(w',k)-A(w,k)$. Taking $w'=y$ and $w=y$ gives
\begin{align*}
T(y,k)\geq  \sup_{z\in [x,y)}\big( T(z,k+1)+A(y,k)-A(z,k)\big).
\end{align*}
Taking $w=y$ and $w'\downarrow y$ gives
\begin{align*}
T(y,k)\geq T(y,k+1).
\end{align*}
Combining the above yields \eqref{equ:polymer-to-lp_3}. The proof is complete.
\end{proof}

\subsection{Directed Polymer to Directed Polymer}

\begin{lemma}\label{lem:polymer_action_Lip}
Let $A$ be an environment on $I\times J$, $w\in (1,\infty)$, $y_0\in I$ and $r\in [0,\infty)$. Then there exists a function $F:(0,1]\times [1,\infty)\times\mathbb{N}\to [0,\infty)$ that depends on the data above and satisfies the following property. Let $T:\mathbb{R}\times\mathbb{N}\to\mathbb{R}\cup\{\pm\infty\}$ be a $w$-polymer $A$-action with $w^{T(y_0,1)-A(y_0,1)}\leq r$. Then for any $(\delta,\alpha,k)\in (0,1]\times [1,\infty)\times J$, we have
\begin{align}\label{equ:T-A_upperbound}
w^{T(y,k)-A(y,k)}\leq F(\delta,\alpha,k)\ \textup{for all}\ y\in [-\alpha+y_0,-\delta+y_0]\cap I.
\end{align}  
\end{lemma}
\begin{proof}
For $k=1$, we set $F(\delta,\alpha,1)\equiv r$. For $k\geq 2$, we inductively define
$$
F(\delta,\alpha,k)= 3\delta^{-1}\left( \sup_{z\in [-\alpha+y_0,y_0]}w^{A(z,k-1)-A(z,k)}\right) F(3^{-1}\delta,\alpha,k-1).
$$
We note that because $z\mapsto A(z,k-1)-A(z,k)$ is cadlag, $\sup_{z\in [-\alpha+y_0,y_0]}(A(z,k-1)-A(z,k))<\infty$. Hence $F(\delta,\alpha,k)\in (0,\infty)$ for all $(\delta,\alpha,k)\in (0,1]\times [1,\infty)\times\mathbb{N}$. Now we check that \eqref{equ:T-A_upperbound} holds. For $k=1$, \eqref{e:EL-ge_1} gives $w^{T(x,1)-A(x,1)}\leq w^{T(y_0,1)-A(y_0,1)}\leq r$ for all $x\leq y_0$. Hence \eqref{equ:T-A_upperbound} holds for $k=1$. Assume $k\geq 2$ and \eqref{equ:T-A_upperbound} holds for $k$ replaced by $k-1$. Fix $\delta\in(0,1]$, $\alpha\in [1,\infty)$, and $z_0\in [-\alpha+y_0, -\delta+y_0]\cap I$. Denote $C_\alpha=\sup_{z\in [-\alpha+y_0,y_0]} \left(w^{A(z,k-1)-A(z,k)}\right).$ Then we have
\begin{align*}
w^{T(z_0,k)-A(z_0,k)}\leq& 3\delta^{-1}\int_{z_0+\frac{\delta}{3}}^{z_0+\frac{2\delta}{3}} w^{T(z,k)-A(z,k)}\, dz\\
= &3\delta^{-1}\int_{z_0+\frac{\delta}{3}}^{z_0+\frac{2\delta}{3}} w^{T(z,k)-A(z,k-1)} w^{A(z,k-1)-A(z,k)}\, dz\\
\leq &3\delta^{-1}C_\alpha w^{T(z_0+\frac{2\delta}{3},k-1)-A(z_0+\frac{2\delta}{3},k-1)}\\
\leq &3\delta^{-1}C_\alpha F(3^{-1}\delta,\alpha,k-1)\\
=&F(\delta,\alpha,k).
\end{align*}

\end{proof}

\begin{lemma}\label{lem:polymer_action_unique}
Let $A$ be an environment on $\mathbb{R}\times\mathbb{N}$ and $w\in (1,\infty)$. Suppose $T_1, T_2$ are two $w$-polymer $A$-actions with $T_1(y,1)=T_2(y,1)<\infty$ for all $y\in\mathbb{R}$. Then $T_1(y,k)=T_2(y,k)$ for all $(y,k)\in\mathbb{R}\times\mathbb{N}$.
\begin{proof}
The assertion is equivalent to 
\begin{equation}\label{equ:polymer_action_unique_1}
w^{T_1(y,k)-A(y,k)}=w^{T_2(y,k)-A(y,k)}\ \textup{for all}\ (y,k)\in\mathbb{R}\times\mathbb{N}.
\end{equation}
We prove \eqref{equ:polymer_action_unique_1} by an induction argument on $k$. The case $k=1$ holds by the assumption $T_1(y,1)=T_2(y,1)$ for all $y\in\mathbb{R}$. Assume $k\geq 2$ and \eqref{equ:polymer_action_unique_1} holds for $k-1$. From \eqref{e:polymer-mc}, we have for all $x<y$
\begin{align*}
&w^{T_1(x,k-1)-A(x,k-1)}+\int_x^y w^{T_1(z,k)-A(z,k)}w^{A(z,k)-A(z,k-1)}\, dz\\
=&w^{T_2(x,k-1)-A(x,k-1)}+\int_x^y w^{T_2(z,k)-A(z,k)}w^{A(z,k)-A(z,k-1)}\, dz.
\end{align*}
From the induction hypothesis and Lemma~\ref{lem:polymer_action_Lip}, $w^{T_1(x,k-1)-A(x,k-1)}=w^{T_2(x,k-1)-A(x,k-1)}\in [0,\infty)$. Hence
\begin{align*}
\int_x^y w^{T_1(z,k)-A(z,k)}w^{A(z,k)-A(z,k-1)}\, dz=\int_x^y w^{T_2(z,k)-A(z,k)}w^{A(z,k)-A(z,k-1)}\, dz.
\end{align*}
From Lemma~\ref{lem:polymer_action_Lip} and \eqref{e:polymer-mc}, $w^{T_1(z,k)-A(z,k)}$ and $w^{T_2(z,k)-A(z,k)}$ are locally Lipschitz functions in $z$. Then the above implies $w^{T_1(z,k)-A(z,k)}w^{A(z,k)-A(z,k-1)}=w^{T_2(z,k)-A(z,k)}w^{A(z,k)-A(z,k-1)}$ for all $z\in\mathbb{R}$. This gives \eqref{equ:polymer_action_unique_1} and completes the proof.
\end{proof} 
\end{lemma}

\begin{proposition}\label{p:polymer-to-polymer}
Let $ \alpha_n\uparrow\infty$ be a sequence in $[0,\infty]$, $j_n\uparrow\infty$ be a sequence in $\mathbb{N}\cup\{\infty\}$, and $w\in (1,\infty)$. Let $M_n=[-\alpha_n,\infty)\times \{k\in\mathbb{N}\ :\ k<j_n\}$. For each $n$, let $A_n:M_n\to\mathbb{R}$ be an environment and let $T_n:M_n\to\mathbb{R}\cup\{\pm\infty\}$ be a $w$-polymer $A_n$-action.

Suppose that there exists a continuous environment $A:\mathbb{R}\times\mathbb{N}\to\mathbb{R}$ such that $A_n$ converges to $A$ locally uniformly. Further assume that for all $y\in\mathbb{R}$, $T_n(y,1)$ converges to a number in $[-\infty,\infty)$.

Then for each $(y,k)\in \mathbb{R}\times\mathbb{N}$, $T_n(y,k)$ converges to some limit in $[-\infty,\infty)$, denoted by $T(y,k).$ Moreover, $T$ is an upper semicontinuous $w$-polymer $A$-action.  
\end{proposition}

\begin{proof} 
In view of Lemma~\ref{lem:polymer_action_Lip} and \eqref{e:polymer-mc}, $w^{T_n(y,k)-A_n(y,k)}$ are locally bounded and locally Lipschitz with bounds uniform in $n$. Therefore there exists a subsequence (still denoted by $n$) such that $w^{T_n(y,k)-A_n(y,k)}$ converges locally uniformly and we denote the limit by $w^{T(y,k)-A(y,k)}$. From the uniform convergence, it is direct to check \eqref{e:polymer-mc} and hence $T$ is a $w$-polymer $A$-action. From Lemma~\ref{lem:polymer_action_Lip} and \eqref{e:polymer-mc}, $w^{T(y,k)-A(y,k)}$ is locally Lipschitz. Hence $T$ is upper semicontinuous. Since $A_n$ converges to $A$ pointwisely, $T_n$ also converges to $T$ pointwisely.  From Lemma~\ref{lem:polymer_action_unique}, the limit $T$ does not depend on the subsequence. The proof is complete.
\end{proof}

\subsection{Discrete actions}

In this subsection, we define actions in a discrete environment and record the convergence results.

Fix $\delta>0$. A discrete underlying space is of the form $M_\delta= I_\delta\times J$, where $I_\delta=[\alpha,\infty)\cap \delta\mathbb{Z}$ or $I_\delta=\delta\mathbb{Z}$ and $J=\{k\in\mathbb{N}\, :\, k<j\}$ for some $j\in\mathbb{N}\cup\{\infty\}$. A {discrete} environment is a function $A_\delta:M_\delta\to \mathbb R$.

For the next definitions, we use the convention that $T_\delta(z,k+1)\equiv -\infty$ if $k+1\notin J$.

\begin{definition}
Let $A_\delta$ be a discrete environment on $M_\delta=I_\delta\times J$, and let $T_\delta :M_\delta\to \mathbb{R}\cup\{\pm\infty\}$.
We call $T_\delta$ a {\bf $\delta$-discrete $A_\delta$-{action}} if for all $k\in J$ and $y,y-\delta\in I_\delta$ we have
\begin{align}\label{equ:Local_metric_composition_law_discrete0}
T_\delta(y,k) = (T_\delta(y-\delta,k)  \vee T_\delta(y,k+1))+ A_\delta(y,k)-A_\delta(y-\delta,k)
\end{align}
\end{definition}
This equation describes the evolution of discrete last passage percolation with weights given by $A_\delta(y,k)-A_\delta(y-\delta,k)$. It follows by induction on $(y-x)/\delta$ that  \eqref{equ:Local_metric_composition_law_discrete0} implies that for $k\in J$ and $x,y\in I_\delta$ with $x<y$, we have
\begin{align}\label{equ:Local_metric_composition_law_discrete}
T_\delta(y,k) =(T_\delta(x,k)+A_\delta(y,k)-A_\delta(x,k) ) \vee \sup_{z\in (x,y]\cap  \delta\mathbb{Z} }\big( T_\delta(z,k+1)+A_\delta(y,k)-A_\delta(z-\delta,k)\big).
\end{align}
 which reduces to \eqref{equ:Local_metric_composition_law_discrete0} when  $x=y-\delta$. 
\begin{definition}
Let $A_\delta$ be a discrete environment on $M_\delta=I_\delta\times J$ and let  $T_\delta :M_\delta\to \mathbb{R}\cup\{\pm\infty\}$, and let $w>1$.

We call $T_\delta$ a {\bf $\delta$-discrete, $w$-polymer $A_\delta$-action} if  for all $k\in J$ and $y,y-\delta\in I_\delta$ it satisfies
\begin{align}\label{equ:Local_metric_composition_law_discrete_polyer0}
w^{T_\delta(y,k)} =(w^{T_\delta(x,k)}+ w^{  T_\delta(z,k+1)})
w^{A_\delta(y,k)-A_\delta(y-\delta,k)}
\end{align}
\end{definition}
\noindent This equation describes a discrete directed polymer free energy evolution with log weights given by $A_\delta(y,k)-A_\delta(y-\delta,k)$. It follows by induction on $(y-x)/\delta$ that  \eqref{equ:Local_metric_composition_law_discrete_polyer0} implies that for $k\in J$ and $x,y\in I_\delta$ with $x<y$, we have
\begin{align}\label{equ:Local_metric_composition_law_discrete_polyer}
w^{T_\delta(y,k)} =w^{T_\delta(x,k)+A_\delta(y,k)-A_\delta(x,k) }+ \sum_{z\in (x,y] \cap  \delta\mathbb{Z} } w^{  T_\delta(z,k+1)+A_\delta(y,k)-A_\delta(z-\delta,k)},
\end{align}
 which reduces to \eqref{equ:Local_metric_composition_law_discrete_polyer0} when  $x=y-\delta$. 

\begin{lemma}\label{lem:discrete_EL}
Let $A_\delta$ be an environment defined on $M_\delta=I_\delta\times J$, $w>1$ and let $T_\delta$ be a $\delta$-discrete, $w$-polymer $A_\delta$-action with $T_\delta(y,k)<\infty$ for all $(y,k)\in M_\delta$.

Let $I$ be the convex hull of $I_\delta$. Extend $A_\delta$ and $T_\delta$ to $I\times J$ through linear interpolation and denote the extensions by $\overline{A}_\delta$ and $\overline{T}_\delta$ respectively. Then for all $x<y$ in $I$ and $k\in J$, we have
\begin{equation}\label{equ:discrete_EL-le}
\begin{split}
\overline{T}_\delta(y,k)\le &\sup_{x'\in [x-\delta,x+\delta]\cap I}(\overline{T}_\delta(x',k)+\overline{A}_\delta(y,k)-\overline{A}_\delta(x',k))\\
&\vee
\sup_{z\in [x-\delta,y+\delta]\cap I}(\overline{T}_\delta(z,k+1)+\overline{A}_\delta(y,k)-\overline{A}_\delta(z-\delta,k))+\log_w(2+\delta^{-1}(y-x)).
\end{split}
\end{equation}
\begin{equation}\label{equ:discrete_EL-ge_1}
\overline{T}_\delta(y,k)\ge \overline{T}_\delta(x,k)+\overline{A}_\delta(y,k)-\overline{A}_\delta(x,k).
\end{equation}
\begin{equation}\label{equ:discrete_EL-ge_2}
\overline{T}_\delta(y,k)\ge \sup_{z\leq y, z-\delta\in I}(\overline{T}_\delta(z,k+1)+\overline{A}_\delta(y,k)-\overline{A}_\delta(z-\delta,k)). 
\end{equation}

If $T_\delta$ is a $\delta$-discrete $A_\delta$-action, then \eqref{equ:discrete_EL-ge_1} and \eqref{equ:discrete_EL-ge_2} also hold and \eqref{equ:discrete_EL-le} holds without the term $\log_w(2+\delta^{-1}(y-x)).$
\end{lemma}
\begin{proof}
From \eqref{equ:Local_metric_composition_law_discrete_polyer}, we have for all $x<y$ in $I_\delta$,
\begin{align*}
 {T}_\delta(y,k)\le & ( {T}_\delta(x,k)+ {A}_\delta(y,k)- {A}_\delta(x,k))\\
 &\vee
\sup_{z\in (x,y]\cap I_\delta}( {T}_\delta(z,k+1)+ {A}_\delta(y,k)- {A}_\delta(z-\delta,k))+\log_w(1+\delta^{-1}(y-x)),
\end{align*}
\begin{equation*} 
 {T}_\delta(y,k)\ge  {T}_\delta(x,k)+ {A}_\delta(y,k)- {A}_\delta(x,k),
\end{equation*}
and
\begin{equation*} 
 {T}_\delta(y,k)\ge \sup_{z\leq y, z-\delta\in I_\delta}( {T}_\delta(z,k+1)+ {A}_\delta(y,k)- {A}_\delta(z-\delta,k)). 
\end{equation*}
From these inequalities, it is direct to check \eqref{equ:discrete_EL-le}-\eqref{equ:discrete_EL-ge_2}.
\end{proof}

\begin{proposition}\label{p:polymer_to_lp_discrete}
Let $\delta_n\downarrow 0$ be a sequence in $(0,\infty)$, $ \alpha_n\uparrow\infty$ be a sequence in $[0,\infty]$, $j_n\uparrow\infty$ be a sequence in $\mathbb{N}\cup\{\infty\}$, and $w_n\uparrow\infty$ be a sequence in $(1,\infty)$ such that $\log_{w_n}\delta_n\to 0$. Let $M_{\delta_n}=([-\alpha_n,\infty)\cap \delta_n\mathbb{Z})\times \{k\in\mathbb{N}\ :\ k<j_n\}$. For each $n$, let $A_{\delta_n}:M_{\delta_n}\to\mathbb{R}$ be a discrete environment and let $T_{\delta_n}:M_{\delta_n}\to\mathbb{R}\cup\{-\infty\}$ be a $\delta_n$-discrete, $w_n$-polymer $A_{\delta_n}$-action. 

Let $\overline{A}_{\delta_n}$ and $\overline{T}_{\delta_n}$ be the extensions of ${A}_{\delta_n}$ and ${T}_{\delta_n}$ through linear interpolation respectively. Suppose that there exists a continuous environment $A:\mathbb{R}\times\mathbb{N}\to\mathbb{R}$ such that $\overline{A}_{\delta_n}$ converges to $A$ locally uniformly. Further assume $\overline{T}_{\delta_n}$ converges to $T$ in $\uc(\mathbb{R}\times\mathbb{N})$. Then $T$ is an $A$-action. The same assertion holds if $T_{\delta_n}$ are $\delta_n$-discrete $A_{\delta_n}$ actions.  
\end{proposition}
The proof of this proposition is similar to the one of Proposition~\ref{p:polymer_to_lp}, with Lemma~\ref{l:EL} replaced by Lemma~\ref{lem:discrete_EL}. We omit the details.

\section{Airy line ensemble actions}\label{sec:Airyaction}

In this section we establish properties that are specific to the Airy line ensemble. The goal is to prove Theorems~\ref{thm:action-rep-sym} and \ref{t:action-rep_discrete}, a general framework to establish convergence to the Airy sheet. 

In Subsection~\ref{sec:proof_of_main}, we prove Theorem~\ref{t:main}, Corollary~\ref{c:main} and Corollary~\ref{c:main2} using Propositions~\ref{prop:main_1} and \ref{prop:main_2}. These two propositions are proved in Subsection~\ref{sec:proof_of_4344}. Lastly, we introduce and prove Theorems~\ref{thm:action-rep-sym} and \ref{t:action-rep_discrete} in Subsection~\ref{sec:conv_to_AS}.

In the following, we collect some important properties of the Airy line ensemble and the Airy sheet for later use. Recall that $\mathcal{A}$ is the parabolic Airy line ensemble and that $\cS$ is the Airy sheet. It is proved in \cite{DOV} that there is a natural coupling of $\mathcal{A}$ and $\cS$. Under such a coupling, almost surely one has for all $x>0,y_1, y_2\in\mathbb{R}$,
\begin{equation}\label{equ:D}
\begin{split}
\lim_{k\to\infty} \mathcal{A}\left( (-\sqrt{k/(2x)}  ,k )\ar  (y_2,1)\right) -\mathcal{A}\left( (-\sqrt{k/(2x)}  ,k )\ar  (y_1,1)\right) =\mathcal{S}(x,y_2)-\mathcal{S}(x,y_1).
\end{split} 
\end{equation}
 
\begin{proposition}[Corollary 10.7 in \cite{DOV}]\label{prop:DOV}
There exists a non-negative random variable $C$ satisfying $\mathbb{P}(C>m)\leq c e^{-dm^{3/2}}$ for universal constants $c,d$ and all $m>0$ such that
\begin{align}\label{equ:Airy_sheet_parabolic}
\left|\cS(x,y)+(x-y)^2 \right|\leq C\left(\log(|x|+|y|+2)\right)^2\ \textup{for all}\ x,y\in\mathbb{R}.
\end{align}
\end{proposition}

Recall that $\uc_+$ is the set of all upper semicontinuous functions $f:[0,\infty)\to \mathbb R\cup\{-\infty\}$. Set
\begin{equation}\label{e:Sf_1}
\cS(f,y)=\sup_{x\ge 0}[ f(x)+ \cS(x,y)].
\end{equation}
The next proposition is essentially Proposition 6.1 in \cite{sarkar2021brownian}. We rephrase the statement for our purpose and provide a proof for the reader's convenience.

\begin{proposition}[Proposition 6.1 in \cite{sarkar2021brownian}]\label{prop:SV}
Fix a realization such that \eqref{equ:Airy_sheet_parabolic} holds for some $C\in [0,\infty)$. Then the following statements hold. First, a function $f\in \uc_+$ with $f\not\equiv-\infty$ satisfies $\mathcal{S}(f,y)\in\mathbb{R}$ for all $y$ if and only if 
\begin{equation}\label{equ:finitary}
\lim_{x\to \infty} \frac{f(x)-x^2}{x}=-\infty.
\end{equation}
Second, $\mathcal{S}(f,y)\in\mathbb{R}$ for all $y$ implies $\cS(f,y)$ is a continuous function on $y\in\mathbb{R}$.
\begin{proof}
Fix a realization such that \eqref{equ:Airy_sheet_parabolic} holds for some $C\in [0,\infty)$. Fix an arbitrary $f\in\uc_+$ with $f\not\equiv-\infty$ that satisfies \eqref{equ:finitary}. Since $f\not\equiv -\infty$, we have $\cS(f,y)>-\infty$ for all $y\in\mathbb{R}$. Moreover,
\begin{align*}
\cS(f,y)\leq \sup_{x\geq 0}\left( f(x)-x^2+2xy-y^2+C\left(\log(|x|+|y|+2)\right)^2 \right).
\end{align*}
It is direct to see that for any $L_1>0$ there exists $L_2>0$ such that for all $|y|\leq L_1$,
\begin{align*}
\cS(f,y)=\sup_{x\in [0,L_2]} f(x)+\cS(x,y).
\end{align*}
Hence $\cS(f,y)<\infty$ and is continuous in $y$.

On the other hand, let $f\in \uc_+$ be a function such that \eqref{equ:finitary} fails. There exists $a\in\mathbb{R}$ and a sequence $x_n\to \infty$ such that $f(x_n)-x_n^2\geq a x_n $. Then for any $y\geq (-a+1)/2$, we have
\begin{align*}
\cS(f,y)\geq &\limsup_{n\to\infty}\left( f(x_n)-x_n^2+2x_ny-y^2-C\left(\log(|x_n|+|y|+2)\right)^2 \right)\\
\geq &\limsup_{n\to\infty}\left( |x_n| -y^2-C\left(\log(|x_n|+|y|+2)\right)^2 \right)=\infty.\qedhere
\end{align*}
\end{proof}
\end{proposition}

\subsection{Proof of Theorem~\ref{t:main}}\label{sec:proof_of_main}

In this subsection, we prove Theorem~\ref{t:main}, Corollary~\ref{c:main} and Corollary~\ref{c:main2}. In order to prove Theorem~\ref{t:main}, we need Propositions~\ref{prop:main_1} and \ref{prop:main_2} below. The proofs for Propositions~\ref{prop:main_1} and \ref{prop:main_2} are postponed to the next subsection.

\begin{proposition}\label{prop:main_1}
Almost surely, the following statement holds. For any $\mathcal{A}$-action $T$ which is upper semicontinuous and $T(y,1)\in\mathbb{R}$ for all $y\in\mathbb{R}$, there exists $f\in\uc_+$ such that $T(1,y)=\cS(f,y)$ for all $y\in\mathbb{R}$.
\end{proposition}
\begin{proposition}\label{prop:main_2}
Almost surely, the following  statement holds. For any $f\in\uc_+$ such that $\cS(f,y)\in\mathbb{R}$ for all $y\in\mathbb{R}$, there exists an $\mathcal{A}$-action $T$ such that $T(y,1)=\cS(f,y)$ for all $y\in\mathbb{R}$.  
\end{proposition}
\begin{proof}[Proof of Theorem~\ref{t:main}]
Fix a realization such that the statements in Propositions~\ref{prop:SV}, \ref{prop:main_1}, and \ref{prop:main_2} hold. From Proposition~\ref{prop:main_2}, we have
\begin{align*}
&\{\cS(f,\cdot):f\in \uc_+,\cS(f,x)\in \mathbb{R}\ \forall x\in\mathbb{R} \} 
\subset  \{T(1,\cdot):T\mbox{ is an}\ \mathcal{A}\mbox{-action, }T(1,y)\in\mathbb{R}\ \textup{for all}\ y \}.
\end{align*}

We aim to prove the other direction of inclusion. Let $T$ be an $\mathcal{A}$-action with $T(y,1)\in\mathbb{R}$ for all $y\in\mathbb{R}$. Let $T'$ be the upper semicontinuous version of $T$. From Lemma~\ref{lem:UCaction}, $T'$ is an $\mathcal{A}$-action. Since $T(y,1)-\mathcal{A}(y,1)$ is monotone non-decreasing in $y$, and $T(y,1)\in\mathbb{R}$ for all $y$, we have $T'(y,1)\in\mathbb{R}$ for all $y$. From Proposition~\ref{prop:main_1}, there exists $f\in \uc_+$ such that $\cS(f,y)=T'(y,1)$ for all $y$. From Proposition~\ref{prop:SV}, $\cS(f,y)$ is continuous. Hence $T'(y,1)$ is also continuous. In view of the monotonicity of $T(y,1)-\mathcal{A}(y,1)$, we conclude that $T(y,1)=T'(y,1)=\cS(f,y)$ for all $y$. In short, for any $\mathcal{A}$-action $T$ with $T(y,1)\in\mathbb{R}$ for all $y\in\mathbb{R}$, there exists $f\in \uc_+$ such that $\cS(f,y)=T(y,1)$ for all $y\in\mathbb{R}$. This shows the other direction of inclusion and completes the proof. 
\end{proof}

We are now ready to prove the two main corollaries from the introduction. 

\begin{proof}[Proof of Corollary \ref{c:main}]
From Proposition~\ref{prop:DOV}, with probability one, we have
\begin{equation}\label{equ:c_main_1}
\lim_{y\to\pm\infty} \frac{\cS(z,y)-y^2}{y}=2z\ \textup{for all}\ z\in\mathbb{R}.
\end{equation}
Let $T$ be an $\mathcal{A}$-action and $x\in\mathbb{R}$ be a real number as described in Corollary~\ref{c:main}. In particular, with probability one, we have $T(y,1)\in\mathbb{R}$ for all $y\in\mathbb{R}$ and 
\begin{equation}\label{equ:c_main_2}
\liminf_{y\to\infty} \frac{T(y,1)+y^2}{y}\le 2x, \qquad \mbox{and} \qquad \liminf_{y\to-\infty} \frac{T(y,1)+y^2}{-y}\le -2x.
\end{equation}
From Theorem~\ref{t:main}, with probability one, we have 
\begin{equation}\label{equ:c_main_3}
T(1,y)=\cS(f,y)\ \textup{for all}\ y\in\mathbb{R} 
\end{equation}
for some $f\in \uc_+$. 

We aim to show that when \eqref{equ:c_main_1}, \eqref{equ:c_main_2}, and \eqref{equ:c_main_3} all occur, $T(1,y)-\cS(x,y)$ does not depend on $y$. Then setting $C=T(1,0)-\cS(x,0)$ yields the desired result. 

From now on, we fix a realization such that \eqref{equ:c_main_1}, \eqref{equ:c_main_2}, and \eqref{equ:c_main_3} all occur. Let $x_0\geq 0$ be a point at which $f(x_0)>-\infty$. Note that from the definition of $\uc_+$, $f(x_0)\in\mathbb{R}$. Combining \eqref{equ:c_main_1}, \eqref{equ:c_main_2} and \eqref{equ:c_main_3}, we have
\begin{align*}
2x\geq &\liminf_{y\to\infty} \frac{T(y,1)+y^2}{y} 
=\liminf_{y\to\infty} \frac{\cS(f,y)+y^2}{y}\\
\geq & \liminf_{y\to\infty} \frac{\cS(x_0,y)+f(x_0)+y^2}{y} 
=2x_0.
\end{align*}
A similar argument gives $2x\leq 2x_0$. Hence
\begin{align*}
f(z)=\left\{
\begin{array}{cc}
f(x), & z=x,\\
-\infty, & z\neq x. 
\end{array}
\right.
\end{align*}
This implies $T(1,y)=\cS(f,y)=\cS(x,y)+f(x)$. The proof is complete.
\end{proof}

\begin{lemma}\label{l:ergodic}
Let $(Y_n, \,n\in \mathbb Z)$ be a stationary ergodic process, and let $Y'_n=Y_n+X$ for some random variable $X$. Assume $Y_n\ed Y'_n$\, for all $n$. Then $X=0$ almost surely. 
\end{lemma}
\begin{proof}
Let $\lambda \in \mathbb R$. By the ergodic theorem, 
$$
\lim_{n\to\infty} \frac{1}{n}\sum_{j=1}^n e^{i\lambda Y'_j}=e^{i\lambda X} \lim_{n\to\infty} \frac{1}{n}\sum_{j=1}^n e^{i\lambda Y_j}=e^{i\lambda X}\mathbb{E}\left[ e^{i\lambda Y_1}\right].
$$
almost surely and in $L^1$. Taking expectations, and using the bounded convergence theorem, we deduce 
$$
\mathbb{E}\left[ e^{i\lambda Y_1}\right]=\mathbb{E}\left[ e^{i\lambda X}\right]\mathbb{E}\left[ e^{i\lambda Y_1}\right].
$$
Since $\varphi(\lambda)=\mathbb{E}\left[ e^{i\lambda Y_1}\right]$ is continuous in $\lambda$ and $\varphi(0)=1$. There is an open neighborhood $U$ of $0$ such that $\varphi(\lambda)\neq 0$ for $\lambda\in U$. Then we have $\mathbb{E}\left[ e^{i\lambda X}\right]=1 $ for all $\lambda\in U$. This implies $X=0$, see Exercise 3.3.19 in \cite{durrett}.
\end{proof}

\begin{proof}[Proof of Corollary \ref{c:main2}]
Let $T$ be an $\mathcal{A}$-action and $x\in\mathbb{R}$ be a real number as described in Corollary~\ref{c:main2}. We start by checking that almost surely $T(y,1)\in\mathbb{R}$ for all $y\in\mathbb{R}$. Since $T(n,1)+(n-x)^2\ed \textup{TW}$ for $n\in\mathbb{Z}$, with probability one, we have $T(n,1)\in\mathbb{R}$ for all $n\in\mathbb{Z}$. From the action recursion \eqref{equ:Local_metric_composition_law}, we deduce 
\begin{align*}
T(\lfloor y \rfloor,1)+\mathcal{A}(y,1)-\mathcal{A}(\lfloor y \rfloor,1) \leq  T(y,1)\leq T(\lceil y \rceil,1)-\mathcal{A}(\lceil y \rceil,1)+\mathcal{A}(y,1)
\end{align*}
for all $y\in\mathbb{R}$. Hence with probability one, $T(y,1)\in\mathbb{R}$ for all $y\in\mathbb{R}$. 

We claim that almost surely
\begin{equation}\label{equ:c_main_4}
\liminf_{y\to\infty} \frac{T(y,1)+y^2}{y}\le 2x, \qquad \mbox{and} \qquad \liminf_{y\to-\infty} \frac{T(y,1)+y^2}{-y}\le -2x.
\end{equation}
Assume for a moment \eqref{equ:c_main_4} holds true. From Corollary~\ref{c:main}, there exists a scalar random variable $C$ such that $T(y,1)=\cS(x,y)+C$ almost surely. Since $\cS(x,y)+(y-x)^2$ is the stationary Airy process, it is ergodic, see \cite{prahofer2002scale}. Set $Y_n=\cS(x,n)+(n-x)^2$ and $Y'_n=T(n,1)+(n-x)^2$ for $n\in\mathbb{Z}$. From Lemma \ref{l:ergodic}, we deduce $C=0$ almost surely. Hence $T(y,1)=\cS(x,y)$ almost surely, which is the desired assertion.

It remains to prove \eqref{equ:c_main_4}. We compute
\begin{align*}
\liminf_{y\to\infty} \frac{T(y,1)+y^2}{y}\leq \liminf_{n\to\infty} \frac{T(n,1)+n^2}{n}=2x+\liminf_{n\to\infty} \frac{Y'_n}{n}\leq 2x+\liminf_{n\to\infty} \frac{(Y'_n)_+}{n}.
\end{align*}
Here $(Y'_n)_+$ is the positive part of $Y'_n$. Note that each $Y'_n\ed\textup{TW}$. In particular, $\mathbb{E}[(Y'_n)_+]\in\mathbb{R}$ and does not depend on $n$. Using Fatou's lemma, we get
\begin{align*}
\mathbb{E}\left[\liminf_{n\to\infty}\frac{(Y'_n)_+}{n}\right]\leq \liminf_{n\to\infty}\mathbb{E}\left[\frac{(Y'_n)_+}{n}\right]=0.
\end{align*}
Hence almost surely $\liminf_{n\to\infty}\frac{(Y'_n)_+}{n}=0$ and $\liminf_{y\to\infty} \frac{T(y,1)+y^2}{y}\leq 2x$. A similar argument yields $\liminf_{y\to-\infty} \frac{T(y,1)+y^2}{-y}\leq -2x$ almost surely. These yield \eqref{equ:c_main_4} and complete the proof.
\end{proof}

\subsection{Proofs of Propositions~\ref{prop:main_1} and \ref{prop:main_2}}\label{sec:proof_of_4344}

\begin{proposition}\label{prop:geodesic_monotone}
Fix a realization such that \eqref{equ:D} holds. Let $\gamma$ be an infinite geodesic with respect to $\mathcal{A}$ and $x>0$. Let $q_j=(y_j,k_j)$ be a sequence as described in Definition~\ref{def:Action_from_geodesic}.
\begin{itemize}
\item If $y_j\geq -\sqrt{k_j/(2x)}$ for $j$ large enough. Then for all $z_1\leq z_2$,
\begin{equation}\label{equ:nondecreasing}
T_\gamma(z_2,1)-\cS(x,z_2)\geq T_\gamma(z_1,1)-\cS(x,z_1). 
\end{equation}
\item If $y_j\leq -\sqrt{k_j/(2x)}$ for $j$ large enough. Then for all $z_1\leq z_2$,
\begin{equation}\label{equ:nonincreasing}
T_\gamma(z_2,1)-\cS(x,z_2)\leq T_\gamma(z_1,1)-\cS(x,z_1). 
\end{equation}
\end{itemize}
\end{proposition}

\begin{proof}
Assume $y_j\geq -\sqrt{k_j/(2x)}$ for $j$ large enough. From the definition of infinite geodesics, we have $\lim_{j\to\infty} k_j=\infty$. Set $p_j=(-\sqrt{k_j/(2x)},k_j)$. Fix $z_1\leq z_2$. From the quadrangle inequality \eqref{equ:quadrangle_LP}, we have
\begin{align*}
\mathcal{A}(q_j\ar (z_2,1))-\mathcal{A}(p_j\ar (z_2,1))\geq \mathcal{A}(q_j\ar (z_1,1))-\mathcal{A}(p_j\ar (z_1,1))
\end{align*} 
for $j$ large enough. Rearranging the terms, we have
\begin{align*}
\mathcal{A}(q_j\ar (z_2,1))-\mathcal{A}(q_j\ar q_\gamma)-\big(\mathcal{A}(p_j\ar (z_2,1))-\mathcal{A}(p_j\ar (z_1,1))\big) \geq \mathcal{A}(q_j\ar (z_1,1))-\mathcal{A}(q_j\ar q_\gamma).
\end{align*}
Sending $j\to\infty$ and using \eqref{equ:D}, we get
\begin{align*}
T_\gamma(z_2,1)-\big(\mathcal{S}(x,z_2)-\mathcal{S}(x,z_1) \big)\geq T_\gamma(z_1,1).
\end{align*}
This gives \eqref{equ:nondecreasing}.

Next, we turn to the case that $y_j\leq -\sqrt{k_j/(2x)}$ for $j$ large enough. If $\lim_{j\to\infty}k_j=\infty$, then an argument similar to the above one gives \eqref{equ:nonincreasing} and we omit the details. It remains to consider the situation in which $\lim_{j\to\infty}k_j=k_*<\infty$. In this case, we have $k_j=k_*$ for $j$ large enough and $\lim_{j\to\infty}y_j=-\infty$. Fix $z_1\leq z_2$. We claim that for any $k\geq k_*$, we have
\begin{align}\label{equ:nonincreasing_2}
T_\gamma(z_2,1)-\mathcal{A}((-\sqrt{k/(2x)},k)\ar (z_2,1))\leq T_\gamma(z_1,1)-\mathcal{A}((-\sqrt{k/(2x)},k)\ar (z_1,1)).
\end{align}
Taking $k\to\infty$ in \eqref{equ:nonincreasing_2} and using \eqref{equ:D} yield \eqref{equ:nonincreasing}. Hence it remains to prove \eqref{equ:nonincreasing_2}. From now on, we fix $k\geq k_*$. Let $q'=(z',k_*)$ be a point such that 
$$
\mathcal{A}((-\sqrt{k/(2x)},k)\ar (z_1,1))=\mathcal{A}(-\sqrt{k/(2x)},k)\ar q')+\mathcal{A}(q'\ar (z_1,1)).
$$
From the triangle inequality \eqref{equ:triangle},
\begin{align*}
&\mathcal{A}((-\sqrt{k/(2x)},k)\ar (z_2,1))-\mathcal{A}((-\sqrt{k/(2x)},k)\ar (z_1,1))\\
\geq &\big(\mathcal{A}(-\sqrt{k/(2x)},k)\ar q')+\mathcal{A}(q'\ar (z_2,1))\big)-\big(\mathcal{A}(-\sqrt{k/(2x)},k)\ar q')+\mathcal{A}(q'\ar (z_1,1))\big)\\
=&\mathcal{A}(q'\ar (z_2,1))-\mathcal{A}(q'\ar (z_1,1)).
\end{align*}
For $j$ large enough such that $y_j\leq z'$ and $k_j=k_*$, we have from the quadrangle inequality \eqref{equ:quadrangle_LP}
\begin{align*}
\mathcal{A}(q'\ar (z_2,1))-\mathcal{A}(q'\ar (z_1,1))\geq \mathcal{A}(q_j\ar (z_2,1))-\mathcal{A}(q_j\ar (z_1,1)).
\end{align*}
In short,
\begin{align*}
\mathcal{A}((-\sqrt{k/(2x)},k)\ar (z_2,1))-\mathcal{A}((-\sqrt{k/(2x)},k)\ar (z_1,1))\geq \mathcal{A}(q_j\ar (z_2,1))-\mathcal{A}(q_j\ar (z_1,1))
\end{align*}
for $j$ large enough. Sending $j\to\infty$ yields \eqref{equ:nonincreasing_2}.
\end{proof}


\begin{corollary}\label{c:direction}
Fix a realization such that \eqref{equ:D} holds and \eqref{equ:Airy_sheet_parabolic} holds for some $C\in[0,\infty)$. Let $\gamma$ be an infinite geodesic with respect to $\mathcal{A}$ which ends at $(z_\gamma,1)$. Then
\begin{align*}
\sup_{q\in \gamma\cap (\mathbb{R}\times \mathbb{N})}\inf_{\substack{(y,k)\in \gamma\cap (\mathbb{R}\times \mathbb{N}) \\ (y,k)\preceq q }}\frac{k}{2y^2}=\inf_{q\in \gamma\cap (\mathbb{R}\times \mathbb{N})}\sup_{\substack{(y,k)\in \gamma\cap (\mathbb{R}\times \mathbb{N}) \\ (y,k)\preceq q }}\frac{k}{2y^2}. 
\end{align*}
Denote the above number by $x$. Suppose $x\in [0,\infty)$, then
\begin{align}\label{equ:T_gamma_formula}
T_\gamma(y,1)=\cS(x,y)-\cS(x,z_\gamma)\ \textup{for all}\  y\in\mathbb{R}.
\end{align}
Moreover, $x\in [0,\infty)$ holds provided there exists $z_0>z_\gamma$ such that $T_\gamma(z_0,1)<\infty$. 
\end{corollary}
We call $x$ the {\bf direction} of the infinite geodesic $\gamma$.

\begin{proof} Let
\begin{align*}
\underline x=\sup_{q\in \gamma\cap (\mathbb{R}\times \mathbb{N})}\inf_{\substack{(y,k)\in \gamma\cap (\mathbb{R}\times \mathbb{N}) \\ (y,k)\preceq q }}\frac{k}{2y^2},\ \bar x=\inf_{q\in \gamma\cap (\mathbb{R}\times \mathbb{N})}\sup_{\substack{(y,k)\in \gamma\cap (\mathbb{R}\times \mathbb{N}) \\ (y,k)\preceq q }}\frac{k}{2y^2}.
\end{align*}
Suppose $\underline{x}<\bar{x}$. Fix arbitrary $x'\in (\underline{x},\bar{x})\subset (0,\infty)$. From Proposition~\ref{prop:geodesic_monotone}, $T_\gamma(z,1)-S(x',z)$ does not depend on $z$. From $T_\gamma(z_\gamma,1)=0$, we have $T_\gamma(z,1)=S(x',z)-S(x',z_\gamma)$ for all $z\in\mathbb{R}$. In particular, for any $x'<x''$ in $ (\underline{x},\bar{x})$ we have $S(x',z)-S(x'',z)\equiv S(x',z_\gamma)-S(x'',z_\gamma)$ for all $z\in\mathbb{R}$. This is impossible in view of \eqref{equ:Airy_sheet_parabolic}. Therefore, we must have $\underline{x}=\bar{x}=:x$.

Next, we prove \eqref{equ:T_gamma_formula}. We separate the discussion depending on $x\in (0,\infty)$ or $x=0$. Assume $x\in (0,\infty)$. From Proposition~\ref{prop:geodesic_monotone}, for any $0<\varepsilon<x$ and $z_2\geq z_1$, we have
\begin{align*}
T_\gamma(z_2,1)-\cS(x-\varepsilon,z_2)\geq T_\gamma(z_1,1)-\cS(x-\varepsilon,z_1) 
\end{align*} 
and
\begin{align*}
T_\gamma(z_2,1)-\cS(x+\varepsilon,z_2)\leq T_\gamma(z_1,1)-\cS(x+\varepsilon,z_1). 
\end{align*} 
Sending $\varepsilon\to 0$ and using $T_\gamma(z_\gamma,1)=0$ yield \eqref{equ:T_gamma_formula}. Now we assume $x=0$. From Proposition~\ref{prop:geodesic_monotone}, for any $z_2\geq z_1$, we have
\begin{align*}
T_\gamma(z_2,1)-\cS(0,z_2)\leq T_\gamma(z_1,1)-\cS(0,z_1). 
\end{align*}  
Using $\cS(0,z)=\mathcal{A}(z,1)$, we get
\begin{align*}
T_\gamma(z_2,1)\leq T_\gamma(z_1,1)+\mathcal{A}(z_2,1)-\mathcal{A}(z_1,1). 
\end{align*}
On the other hand, the action recursion \eqref{equ:Local_metric_composition_law} gives
\begin{align*}
T_\gamma(z_2,1)\geq T_\gamma(z_1,1)+\mathcal{A}(z_2,1)-\mathcal{A}(z_1,1). 
\end{align*}
Hence \eqref{equ:T_gamma_formula} holds true.  

Lastly, we assume there exists $z_0>z_\gamma$ such that $T_\gamma(z_0,1)<\infty$ and show that $x<\infty$. Suppose $x=\infty$. From \eqref{equ:nondecreasing}, we have
\begin{align*}
T_\gamma(z_0,1)=T_\gamma(z_0,1)-T_\gamma(z_\gamma,1)\geq \mathcal{S} (x',z_0)-\mathcal{S} (x',z_\gamma)
\end{align*}
for all $x'>0$. Sending $x'\to \infty$ and applying \eqref{equ:Airy_sheet_parabolic}, we have $T_\gamma(z_0,1)=\infty$, which contradicts the assumption.
\end{proof}

\begin{proof}[Proof of Proposition \ref{prop:main_1}]
Fix a realization such that \eqref{equ:D} holds and \eqref{equ:Airy_sheet_parabolic} holds for some $C\in [0,\infty)$. Let $T$ be an upper semicontinuous $\mathcal{A}$-action with $T(y,1)\in\mathbb{R}$ for all $y\in\mathbb{R}$. From Proposition~\ref{prop:Bq_geodesic}, for all $z\in\mathbb{R}$ there exists an infinity $(T,(z,1))$-geodesic $\gamma(z)$. From Lemma \ref{l:decomp}, we have for any $z,z'\in \mathbb{R}$,
\begin{align*}
T_{\gamma(z)}(z',1)\leq T(z',1)-T(z,1)<\infty.
\end{align*}
From Corollary~\ref{c:direction}, the direction of $\gamma(z)$, denoted by $d(z)$, belongs to $[0,\infty)$. From Proposition~\ref{prop:decomposition} and Corollary~\ref{c:direction},
\begin{align*}
T(y,1)=\sup_{z\in\mathbb{R}}\big( T_{\gamma(z)}(y,1)+T(z,1) \big)=\sup_{z\in\mathbb{R}}\big( \cS(d(z),y)-\cS(d(z),z) +T(z,1) \big).
\end{align*}
Breaking down the supremum according to the values of $d(z)$, we get
\begin{align*}
T(y,1)=&\sup_{x\geq 0}\sup_{z:d(z)=x} \big( \cS(x,y)-\cS(x,z) +T(z,1) \big)\\
=&\sup_{x\geq 0}\left(\cS(x,y)+\sup_{z:d(z)=x} \big( -\cS(x,z) +T(z,1) \big)\right).
\end{align*} 
Here we adopt the convention that the supremum of an empty set is $-\infty$. Set the function $g(x)$ by
\begin{align*}
g(x)=\sup_{z:d(z)=x} \big( -\cS(x,z) +T(z,1) \big).
\end{align*}
Then we get
\begin{align*}
T(y,1)=\sup_{x\geq 0}\left(\cS(x,y)+g(x) \right).
\end{align*}
Since $\cS$ is continuous, in the left expression $g$ can be replaced by its upper semicontinuous version $f(x)=\limsup_{y\to x} g(x)$, where in the limsup $y=x$ is allowed. Since $T(1,y)\in\mathbb{R}$, $f(x)\in\mathbb R\cup \{-\infty\}$ and $f\not\equiv -\infty$.
\end{proof}
\begin{proof}[Proof of Proposition~\ref{prop:main_2}]
By Lemma 3.4 in \cite{sarkar2021brownian}, almost surely, for all $x\in \mathbb{Q}_+$, there exists an infinite geodesic $\sigma(x)$ which ends at $(0,1)$ and has the direction $x$. Through approximation, we obtain infinite geodesics $\sigma(x)$ ending at $(0,1)$ in direction $x$ for all $x\in [0,\infty)$. From Corollary \ref{c:direction}, $T_{\sigma(x)}(y,1)=\cS(x,y)-\cS(x,0)$. 
Define $T^x$ on $\mathbb{R}\times\mathbb{N}$ by
\begin{align*}
T^x(y,k)=T_{\sigma(x)}(y,k)+\cS(x,0).
\end{align*}
From Lemma~\ref{lem:Bgamma_action}, $T^x$ is an $\mathcal{A}$-action. 
Let 
$$
T(y,k)=\sup_{x\geq 0} (T^x(y,k)+f(x)).
$$ 
From Lemma~\ref{lem:actionsup}, $T$ is an $\mathcal{A}$-action. Moreover,
\begin{align*}
T(y,1)=\sup_{x\geq 0} (T^x(y,1)+f(x))&=\sup_{x\geq 0} (T_{\sigma(x)}(y,1)+\cS(x,0)+f(x))\\&=\sup_{x\geq 0} (\cS(x,y)+f(x))=\cS(f,y). \qedhere
\end{align*}
\end{proof}

\subsection{Convergence to the Airy sheet}\label{sec:conv_to_AS}

In this subsection, we unify and formalize the convergence result that we will apply to last passage and directed polymer models in the coming section. We first informally describe the kind of setup that we tend to encounter. We will see concrete examples in the next section. This will provide background for the rigorous definitions and theorems in the end. 

We usually start with independent random variables indexed by a strip $\mathbb Z \times \{1,\ldots, n\}$. An RSK isometry argument shows that last passage/partition function values from points in $\mathbb{Z}\times \{n\}$ are represented by similar values through a transformed environment. These environments are typically nonintersecting/weakly nonintersecting walks. For the pattern they present when graphed they are called melons. In semidiscrete models, the environment and melon is given by $n$ functions $A:[0,\infty)\times \{1,\ldots n\}\to \mathbb R$. It helps to represent discrete models the same way for unity.

Next, we apply shift and scaling to the melons in which their tops converge to the Airy line ensemble. We denote the scaled melons $Z_n(0,y,k)$. For $x>0$ we let $Z_n(x,y,k)$ be the last passage/partition value from $(-\alpha_n+x,n)$ to $(y,k)$ in the scaled melon. In the convergence to the Airy line ensemble, the point $(-\alpha_n+x,n)$ tends to infinity, and all we remember in the limit is distances from that point. After this informal description, we need a single rigorous definition. 

\begin{definition}\label{def:action_representation}
Fix $\alpha\in[0,\infty]$ and $j\in \mathbb{N}\cup\{\infty\}$ with $j\geq 2$. Let $H=\{(x,y)\in [0,\infty)\times\mathbb{R}\, :\, y\geq x-\alpha\}$ and let $J=\{n\in\mathbb{N}\, :\, n<j\}.$

An {\bf action representation} is a function $Z:H\times J\to \mathbb R\cup\{ \pm\infty\}$ satisfying the following properties:
\begin{itemize}
\item[(Z1)] For each $(y,k)\in [-\alpha,\infty)\times J$, $Z(0,y,k)\in\mathbb{R}$ and is cadlag in $y$.
\item[(Z2)] For any $x>0$, $Z(x,\cdot,\cdot)$ is an upper semicontinuous action with respect to $Z(0,\cdot,\cdot)$.
\item[(Z3)] For each $(x,y)\in H$, $Z(x,y,1)\in\mathbb{R}$ and $Z(x,0,1)$ is cadlag in $x$. 
\item[(Z4)] For any $0\leq x_1\leq x_2$ and $y_1\leq y_2$, we have
\begin{align}\label{equ:quadrangle_action}
Z(x_1,y_1,1)+Z(x_2,y_2,1)\geq Z(x_1,y_2,1)+Z(x_2,y_1,1). 
\end{align} 
\end{itemize}
We say $Z$ is of size $(\alpha,j)$ to emphasize the role of $\alpha$ and $j$. If property (Z2) is replaced by:
\begin{itemize}
\item[(Z2')] For any $x>0$, $Z(x,\cdot,\cdot)$ is an upper semicontinuous $w$-polymer action with respect to $Z(0,\cdot,\cdot)$,
\end{itemize}
we say $Z$ a $w$-polymer action representation.
\end{definition}

We emphasize that we will not use the background story in order to work with action representations: Theorems \ref{thm:action-rep-sym} and \ref{t:action-rep_discrete} will apply to anything that fits this simple definition.  For the KPZ equation, the background story will be  different. 

For the next theorem, recall that $\mathcal S$ is the Airy sheet, $\mathcal A$ is the parabolic Airy line ensemble, and TW is the GUE Tracy-Widom random variable. 

\begin{theorem}\label{thm:action-rep-sym}
Let $ \alpha_n\uparrow\infty$ be a sequence in $[0,\infty]$, $j_n\uparrow\infty$ be a sequence in $\mathbb{N}\cup\{\infty\}$. For $n\in\mathbb{N}$, let $Z_n$ be either a random action representation of size $(\alpha_n,j_n)$, or a random $w_n$-polymer action representation of size $(\alpha_n,j_n)$ such that $w_n\uparrow\infty$. We make the following assumptions.
\begin{itemize}
\item[(Sym1)] For each $n\in\mathbb{N}$ and $(x,y)\in [0,\infty)\times\mathbb{R}$, $Z_n(x,y,1)$ has the same distribution as $Z_n(0,y-x,1)$. 
\item[(Sym2)] For each $n\in\mathbb{N}$, as a process in $x$, $Z_n(x,0,1)$ has the same distribution as $Z_n(0,-x,1)$.
\item[(A1)] Viewed as $\mathbb{D}(\mathbb{R}\times\mathbb{N})$-valued random variables, $Z_n(0,\cdot,\cdot)$ converges in distribution to $\mathcal{A}(\cdot,\cdot)$.  
\end{itemize}
Then there exists a coupling of $\{Z_n\}_{n=1}^\infty$ and the Airy sheet $\cS$ such that almost surely, $Z_n(\cdot,\cdot,1)$ converges locally uniformly to $\cS(\cdot,\cdot)$ on $[0,\infty)\times\mathbb{R}$.
\end{theorem}

Theorem~\ref{thm:action-rep-sym} is a simple consequence of the proposition below.

\begin{proposition}\label{p:action-rep} 
Let $Z_n$ be given as in Theorem~\ref{thm:action-rep-sym}. We make the following assumptions.
\begin{itemize}
\item[(A1)] Viewed as $\mathbb{D}(\mathbb{R}\times\mathbb{N})$-valued random variables, $Z_n(0,\cdot,\cdot)$ converges in distribution to $\mathcal{A}(\cdot,\cdot)$. 
\item[(A2)] For each $(x,y)\in[0,\infty)\times \mathbb{R} $, $Z_n(x,y,1)+(x-y)^2$ converges in distribution to TW.
\item[(A3)] Viewed as $\mathbb{D}([0,\infty))$-valued random variables, $Z_n(\cdot,0,1)$ is tight.   
\end{itemize}
Then there exists a coupling of $\{Z_n\}_{n=1}^\infty$ and the Airy sheet $\cS$ such that almost surely, $Z_n(\cdot,\cdot,1)$ converges locally uniformly to $\cS(\cdot,\cdot)$ on $[0,\infty)\times\mathbb{R}$.
\end{proposition}

\begin{proof}[Proof of Theorem~\ref{thm:action-rep-sym}]
Assumption (A3) of Proposition~\ref{p:action-rep} follows directly from Assumptions (A1) and (Sym2). Similarly, because the Airy line ensemble is continuous and $\mathcal{A}(y,1)+y^2$ has the Tracy-Widom distribution, Assumption (A2) of Proposition~\ref{p:action-rep} is a consequence of Assumptions (A1) and (Sym1). The assertion then follows from applying Proposition~\ref{p:action-rep}.
\end{proof}

Before proving Proposition~\ref{p:action-rep}, we define the discrete version of action representations and state an associated convergence result.

\begin{definition}\label{def:action_representation_discrete}
Fix $\delta>0$, $\alpha\in[0,\infty]$ and $j\in \mathbb{N}\cup\{\infty\}$ with $j\geq 2$. Let 
$$H_\delta=\{(x,y)\in ([0,\infty)\cap\delta_n\mathbb{Z}) \times \delta_n\mathbb{Z} \, :\, y\geq x-\alpha\}$$ 
and let $J=\{n\in\mathbb{N}\, :\, n<j\}.$

A $\delta$-discrete action representation is a function $Z_\delta:H_\delta  \times J\to \mathbb R\cup\{ \pm\infty\}$ satisfying the following properties:
\begin{itemize}  
\item[(Zd1)] For any $x\in (0,\infty)\cap\delta\mathbb{Z}$, $Z_\delta(x,\cdot,\cdot)$ is a $\delta$-discrete action with respect to $Z_\delta(0,\cdot,\cdot)$.
\item[(Zd2)] For each $(x,y)\in H_\delta$, $Z_\delta(x,y,1)\in\mathbb{R}$. 
\item[(Zd3)] For any $0\leq x_1\leq x_2$ and $y_1\leq y_2$ in $\delta\mathbb{Z}$, we have
\begin{align}\label{equ:quadrangle_actiond}
Z_\delta(x_1,y_1,1)+Z_\delta(x_2,y_2,1)\geq Z_\delta(x_1,y_2,1)+Z_\delta(x_2,y_1,1).
\end{align} 
\end{itemize}
\noindent We say $Z$ is of size $(\alpha,j)$ to emphasize the role of $\alpha$ and $j$. If property (Zd1) is replaced by:
\begin{itemize}
\item[(Zd1')] For any $x\in (0,\infty)\cap\delta\mathbb{Z}$, $Z_\delta(x,\cdot,\cdot)$ is a $\delta$-discrete $w$-polymer action with respect to $Z_\delta(0,\cdot,\cdot)$,
\end{itemize}
\noindent we say $Z$ a $\delta$-discrete $w$-polymer action representation.
\end{definition}

\begin{theorem}\label{t:action-rep_discrete} Let $\delta_n\downarrow 0$ be a sequence in $(0,\infty)$, $ \alpha_n\uparrow\infty$ be a sequence in $[0,\infty]$, $j_n\uparrow\infty$ be a sequence in $\mathbb{N}\cup\{\infty\}$. For $n\in\mathbb{N}$, let $Z_{\delta_n}$ be either a $\delta_n$-discrete random action representation of size $(\alpha_n,j_n)$, or a $\delta_n$-discrete random $w_n$-polymer action representation of size $(\alpha_n,j_n)$ such that $w_n\uparrow\infty$ and $\log_{w_n}\delta_n\to 0$. 

Extend $Z_{\delta_n}$ through linear interpolation in the $y$ coordinate and a piecewise constant and cadlag in the $x$ coordinate. Denote the extension by $\overline{Z}_{\delta_n}$. Explicitly,
\begin{align*}
\overline{Z}_{\delta_n}(x,y,k)=\frac{\delta_n(\lfloor\delta_n^{-1} y\rfloor+1)-y}{\delta_n} {Z}_{\delta_n}(\delta_n\lceil\delta_n^{-1} x\rceil,\delta_n\lfloor\delta_n^{-1} y\rfloor,k)\\
+\frac{y-\delta_n\lfloor\delta_n^{-1} y\rfloor}{\delta_n}{Z}_{\delta_n}(\delta_n\lceil\delta_n^{-1} x\rceil,\delta_n(\lfloor\delta_n^{-1} y\rfloor+1),k).
\end{align*}

We make the following assumptions.
\begin{itemize}
\item[(Symd1)] For each $n\in\mathbb{N}$ and $(x,y)\in ([0,\infty)\cap \delta_n\mathbb{Z})\times\delta_n\mathbb{Z}$, $Z_{\delta_n}(x,y,1)$ has the same distribution as $Z_{\delta_n}(0,y-x,1)$. 
\item[(Symd2)] For each $n\in\mathbb{N}$, as a process in $x$, $Z_{\delta_n}(x,0,1)$ has the same distribution as $Z_{\delta_n}(0,-x,1)$. 
\item[(Ad1)] Viewed as $\mathbb{D}(\mathbb{R}\times\mathbb{N})$-valued random variables, $\overline{Z}_{\delta_n}(0,\cdot,\cdot)$ converges in distribution to $\mathcal{A}(\cdot,\cdot)$. 
\end{itemize}
Then there exists a coupling of $\{\overline{Z}_{\delta_n}\}_{n=1}^\infty$ and the Airy sheet $\cS$ such that almost surely, $\overline{Z}_{\delta_n}(\cdot,\cdot,1)$ converges locally uniformly to $\cS(\cdot,\cdot)$ on $[0,\infty)\times\mathbb{R}$.
\end{theorem}
The proof of Theorem~\ref{t:action-rep_discrete} is similar to the one of Theorem \ref{thm:action-rep-sym}. We omit the details.

The rest of the subsection is devoted to proving Proposition~\ref{p:action-rep}. For this purpose, we need the following lemma, which is based on the fact that cadlag functions satisfying the quadrangle inequality are in fact CDFs of measures on $\mathbb R^2$. See the discussion in \cite[Section 4]{dauvergne2021scaling}. The proof is straightforward, so we omit it.

\begin{lemma}\label{l:quadrangle}Let $s_n:[0,\infty)\times \mathbb R\to \mathbb R$ be a sequence of functions that satisfy the quadrangle inequality. Let $s:[0,\infty)\times \mathbb R\to \mathbb R$ be a continuous function. Assume that 
\begin{itemize}
\item $s_n(0,\cdot)$ converges to $s(0,\cdot)$ locally uniformly on $\mathbb{R}$;
\item $s_n(\cdot,0)$ converges to $s(\cdot,0)$ locally uniformly on $[0,\infty)$;
\item For all $(x,y)\in\mathbb{Q}_+\times\mathbb{Q}$, $s_{n}(x,y)$ converges to $s(x,y)$.
\end{itemize}
Then $s_n\to s$ locally uniformly on $[0,\infty)\times\mathbb{R}$.
\end{lemma}

\begin{proof}[Proof of Proposition~\ref{p:action-rep}]
Let $A_n(\cdot,\cdot)=Z_n(0,\cdot,\cdot)$, $S_n(\cdot,\cdot)=Z_n(\cdot,\cdot,1)$, and $T_n^x(\cdot,\cdot)=Z_n(x,\cdot,\cdot)$ for $x>0$. By Lemma~\ref{lem:UC_compact}, for each $x>0$, $T^x(\cdot,\cdot)$ is tight in $\uc(\mathbb{R}\times\mathbb{N})$. All the target spaces, $\mathbb{D}(\mathbb{R}\times\mathbb{N}), \mathbb{R}, \mathbb{D}([0,\infty))$ and $\uc(\mathbb{R}\times\mathbb{N})$, are complete separable metric spaces. By the Skorokhod representation and tightness, we can couple $\{Z_n\}_{n=1}^\infty$ and $\mathcal{A}$ such that almost surely, all of the following statements hold.
\begin{itemize}
\item[(C1)] $A_n$ converges locally uniformly to $\mathcal{A}$ on $\mathbb{R}\times\mathbb{N}$.
\item[(C2)] For all $(x,y)\in [0,\infty)\times \mathbb{R}$ with $(x,y)\in\mathbb{Q}^2$, $S_n(x,y)$ converges to a limit, denoted by $\tilde{\cS}_{x,y}$.
\item[(C3)] $S_n(\cdot,0)$ converges in $\mathbb{D}([0,\infty))$ to a limit, denoted by $\mathcal{R}(\cdot)$.
\item[(C4)] For each $x\in\mathbb{Q}_+$, $T_n^x$ converges in $\uc(\mathbb{R}\times\mathbb{N})$ to a limit, denoted by $T^x$.
\end{itemize}
We will show that under this coupling, almost surely $S_n$ converges locally uniformly to $\cS$ on $[0,\infty)\times\mathbb{R}$. In view of Lemma \ref{l:quadrangle}, it suffices to show that each of the following statements hold almost surely:
\begin{enumerate}
\item[(S1)] $S_n(0,\cdot)$ converges to $\cS(0,\cdot)$ locally uniformly on $\mathbb{R}.$
\item[(S2)] $S_n(\cdot,0)$ converges to $\cS(\cdot,0)$ locally uniformly on $[0,\infty)$.
\item[(S3)] For all $(x,y)\in \mathbb{Q}_+\times \mathbb{Q}$ , $S_n(x,y)$ converges to $\cS(x,y)$.
\end{enumerate}
\noindent{\bf Proof of (S1).}

Since $S_n(0,y)=A_n(y,1)$ and $\cS(0,y)=\mathcal{A}(y,1)$, (S1) follows from (C1).

\noindent{\bf Proof of (S3).}

Now we turn to proving (S3). In view of (C2), it suffices to show that 
\begin{equation}\label{equ:action-rep_0}
\tilde{\cS}_{x,y}=\cS(x,y)\ \textup{for all}\ (x,y)\in\mathbb{Q}_+\times \mathbb{Q}
\end{equation}
holds almost surely. Suppose 
\begin{equation}\label{equ:action-rep_0.5}
T^x(y,1)=\tilde{\cS}_{x,y}\ \textup{for all}\ (x,y)\in\mathbb{Q}_+\times \mathbb{Q}
\end{equation}
holds almost surely. Then for each $x\in\mathbb{Q}_+$, $T^x$ is an $\mathcal{A}$-action (from (C1), (C4), Proposition~\ref{p:lp-to-lp} or Proposition~\ref{p:polymer_to_lp}) with $T^x(y,1)+(y-x)^2\ed\textup{TW}$ (from (A2), (C2) and \eqref{equ:action-rep_0.5}). From Corollary~\ref{c:main2}, almost surely $T^x(y,1)=\cS(x,y)$. Together with \eqref{equ:action-rep_0.5}, this yields \eqref{equ:action-rep_0}.  

In order to prove \eqref{equ:action-rep_0.5}, we claim that almost surely for all $(x,y)\in\mathbb{Q}_+\times \mathbb{Q}$, 
\begin{equation}\label{equ:action-rep_1}
\lim_{y'\in\mathbb{Q}, y'\downarrow y} \tilde{\cS}_{x,y'}-\mathcal{A}(y',1)=\tilde{\cS}_{x,y}-\mathcal{A}(y,1).
\end{equation}
Assume for a moment \eqref{equ:action-rep_1} holds true. We show that (C1), (C2), (C4), and \eqref{equ:action-rep_1} imply \eqref{equ:action-rep_0.5}. Fix arbitrary $x\in\mathbb{Q}_+$. From the quadrangle inequality \eqref{equ:quadrangle_action}, $T_n^x(\cdot,1)-A_n(\cdot,1)$ is monotone nondecreasing. From (C1) and (C4), $T_n^x(\cdot,1)-A_n(\cdot,1)$ converges to $T^x(\cdot,1)-\mathcal{A}(\cdot,1)$ in $\uc(\mathbb{R})$. From (C1) and (C2), for all $y\in\mathbb{Q}$, $T_n^x(y,1)-A_n(y,1)$ converges to $\tilde{\cS}_{x,y}-\mathcal{A}(y,1)$. Together with \eqref{equ:action-rep_1}, we may apply Lemma \ref{lem:UC_no_jump} to conclude \eqref{equ:action-rep_0.5}.

Lastly, we prove \eqref{equ:action-rep_1}. Since $\mathbb{Q}_+\times \mathbb{Q}$ is countable, it suffices to show that \eqref{equ:action-rep_1} holds almost surely for a fixed $(x_0,y_0)\in\mathbb{Q}_+\times \mathbb{Q}$. From the quadrangle inequality, we have almost surely $\tilde{\cS}(x_0,y)-\mathcal{A}(y,1)$ is nondecreasing in $y$. Therefore, it suffices to show
\begin{align}\label{equ:action-rep_2}
\mathbb{E}\left[\lim_{y\in\mathbb{Q}, y\downarrow y_0} \tilde{\cS}_{x_0,y}-\mathcal{A}(y,1)\right]=\mathbb{E}[\tilde{\cS}_{x_0,y_0}-\mathcal{A}(y_0,1)].
\end{align}
Since $\tilde{\cS}_{x,y}+(x-y)^2$ and $\mathcal{A}(y,1)+y^2$ have the same distribution as the GUE Tracy-Widom distribution, we have from the monotone convergence theorem
\begin{align*}
\mathbb{E}\left[\lim_{y\in\mathbb{Q}, y\downarrow y_0} \tilde{\cS}_{x_0,y}-\mathcal{A}(y,1)\right]=&\lim_{y\in\mathbb{Q}, y\downarrow y_0}\mathbb{E}\left[ \tilde{\cS}_{x_0,y}-\mathcal{A}(y,1)\right]=-x_0^2+2x_0y_0.
\end{align*}
It yields \eqref{equ:action-rep_2} and completes the derivation of \eqref{equ:action-rep_1}. The proof for (S3) is complete.
 
\noindent{\bf Proof of (S2).} 
 
We show that (S1), (S3) and (C3) imply (S2). In view of (C3), it suffices to show $\mathcal{R}(x)=\mathcal{S}(x,0)$ for all $x\in [0,\infty)$. From (S1) and (S3), $S_n(x,0)$ converges to $\cS(x,0)$ for all $x\in [0,\infty)\cap \mathbb{Q}$. Together with (C3), we can apply Lemma~\ref{lem:Skorokhod} to conclude that $\mathcal{R}(x)=\mathcal{S}(x,0)$ for all $x\in [0,\infty)$. This finishes the derivation of (S2). The proof of the proposition is complete.
\end{proof}

\begin{remark}
The quadrangle inequality \eqref{equ:quadrangle_action} and the assumption (A3) are only used for joint compact convergence. Without the quadrangle inequality and assuming only (A1) and (A2) in Proposition~\ref{p:action-rep}, we would get a coupling so that $S_n(x, \cdot)\to \cS (x,\cdot)$ locally uniformly for all  $x\in \mathbb [0,\infty)\cap\mathbb{Q}$.
\end{remark}
\section{Applications}\label{sec:application}

In this section, we apply the results in Section~\ref{sec:Airyaction} to prove the convergence to Airy sheet for Brownian last passage percolation, O'Connell-Yor semi-discrete polymer, log-gamma directed polymer, and the KPZ equation. We begin with a simple lemma about change of variables for action representations. The proof is omitted.

\begin{lemma}\label{lem:action_representation_COV}
Let $Z(x,y,k)$ be a $w$-polymer action representation of size $(\alpha,j)$. Fix any positive numbers $a_1,a_2>0$, and real numbers $a_3,a_4, \{a_{5,k}\}_{k\in \mathbb{N}}$ . Then $\bar{Z}(x,y,k)$ defined by
\begin{align*}
\bar{Z}(x,y,k)=\left\{\begin{array}{cc}
a_1Z(0,a_2y+a_3,k)+a_4y+a_{5,k}+(k-1)a_1\log_w a_2, & x=0,\\
a_1Z(a_2x,a_2y+a_3,k)+a_4(y-x)+a_{5,1}+(k-1)a_1\log_w a_2, & x>0.
\end{array}\right.
\end{align*}
is a $w^{\frac{1}{a_1}}$-polymer action representation of size $(a_2^{-1}(\alpha+a_3),j).$ If $Z(x,y,k)$ be an action representation of size $(\alpha,j)$. Then 
\begin{align*}
Z'(x,y,k)= \left\{\begin{array}{cc}
a_1Z(0,a_2y+a_3,k)+a_4y+a_{5,k}, & x=0,\\
a_1Z(a_2x,a_2y+a_3,k)+a_4(y-x)+a_{5,1}, & x>0.
\end{array}\right.
\end{align*}
is an action representation of size $(a_2^{-1}(\alpha+a_3),j).$
\end{lemma}

\subsection{Brownian last passage percolation} 
Let $B:\mathbb{R}\times\mathbb{N}\to\mathbb{R}$ be a random function such that $B(\cdot,k), k\in\mathbb{N}$ are independent two sided Brownian motions with diffusion constant one. For $n\in\mathbb{N}$, $x,y\in\mathbb{R}$ with $y\geq -2^{-1}n^{1/3}+x$, define
\begin{align*}
\cS^{\textup{BL}\ar\textup{FP}}_n(x,y):=n^{1/6}B((2n^{-1/3}x,n)\ar (1+2n^{-1/3}y,1))-2n^{1/3}(y-x)-2n^{2/3}.
\end{align*} 

The next theorem shows that $\cS^{\textup{BL}\to\textup{FP}}_n$ converges to the Airy sheet $\cS$, a result originally proved in \cite{DOV}.

\begin{theorem}\label{t:BLPP_to_FP}
As $C(\mathbb{R}\times\mathbb{R},\mathbb{R})$-valued random variables, $\cS^{\textup{BL}\ar\textup{FP}}_n$ converges in distribution to the Airy sheet $\cS$.
\end{theorem}

\begin{proof}
It suffices to show that for each $r\geq 0$, there exists a coupling of $\cS^{\textup{BL}\ar\textup{FP}}_n|_{[-r,\infty)\times\mathbb{R}}$ and $\cS|_{[-r,\infty)\times\mathbb{R}}$ such that almost surely $\cS^{\textup{BL}\ar\textup{FP}}_n|_{[-r,\infty)\times\mathbb{R}}$ converges to $\cS|_{[-r,\infty)\times\mathbb{R}}$ locally uniformly. From the translation symmetry of $\cS^{\textup{BL}\ar\textup{FP}}_n$ and $\cS$, it suffices to construct the coupling for the case $r=0$.

Let $(WB)_n:[0,\infty)\times \{1,2,\dots, n\}$ be the Brownian $n$-melon constructed from the Brownian motions $B|_{[0,\infty)\times \{1,2,\dots, n\}}$ \cite{DOV} such that
\begin{align*}
(WB)_n(y,1)=B((0,n)\ar (y,1))\ \textup{for all}\ y\geq 0.
\end{align*}
Define $Z_n^{\textup{BL}}(x,y,k)$ for $y\geq x\geq 0$ and $1\leq k\leq n$ by 
\begin{align*}
Z_n^{\textup{BL}}(x,y,k)=\left\{
\begin{array}{cc}
(WB)_n(y,k), & x=0,\\
(WB)_n((x,n)\ar (y,k)), & x>0. 
\end{array}
\right.
\end{align*}
From the RSK isometry \cite{biane2005littelmann,DOV}, we have for $y\geq x\geq 0$,
\begin{align}\label{equ:BL_to_FP_1}
Z_n^{\textup{BL}}(x,y,1)=B((x,n)\ar (y,1)).
\end{align}
From Lemma~\ref{lem:quadrangle_LP}, $Z_n^{\textup{BL}}(\cdot,\cdot,1)$ satisfies the quadrangle inequality. Therefore, $Z_n^{\textup{BL}}$ is an action representation. Define
\begin{align*}
Z_n^{\textup{BL}\ar\textup{FP}}(x,y,k)=n^{1/6} Z_n^{\textup{BL}}(2n^{-1/3}x,1+2n^{-1/3}y,k)-2n^{1/3}(y-x)-2n^{2/3}.
\end{align*}
From Lemma~\ref{lem:action_representation_COV}, $Z_n^{\textup{BL}\ar\textup{FP}}$ is an action representation.  From \eqref{equ:BL_to_FP_1},
\begin{align*}
Z_n^{\textup{BL}\ar\textup{FP}}(x,y,1)=\cS^{\textup{BL}\ar\textup{FP}}_n(x,y).
\end{align*}

Now we apply Theorem~\ref{thm:action-rep-sym}. The symmetry assumptions (Sym1) and (Sym2) follow from \eqref{equ:BL_to_FP_1}. From \cite{CH}, $Z_n^{\textup{BL}\ar\textup{FP}}(0,\cdot,\cdot)$ converges in distribution to the parabolic Airy line ensemble. This verifies the assumption (A1). Then the desired coupling is constructed from Theorem~\ref{thm:action-rep-sym}. The proof is complete. 
\end{proof}

\subsection{O'Connell-Yor polymers}
We begin with defining the polymer free energy in an environment. Let $A$ be an environment on $M=I\times J$. Recall that for any directed path $\gamma$, the length of $\gamma$ with respect to $A$, $L(\gamma,A)$, is defined in \eqref{equ:path_length}. For any $q_1\prec q_2$, the polymer free energy from $q_1$ to $q_2$ is defined by

\begin{align}\label{def:polymer_free_energy}
A(q_1\Rightarrow q_2)=\log\int_{\mathcal{P}(q_1\ar q_2)} e^{L(\gamma,A)}\, d\gamma.
\end{align}
The O'Connell-Yor sheet, $\cS_n^{\textup{OY}}(x,y)$, is the polymer free energy in a Brownian environment.
\begin{align*}
\cS_n^{\textup{OY}}(x,y)=B((x,n)\Rightarrow (y,1)).
\end{align*} 

Next, we introduce the $1:2:3$ scaling of the O'Connell-Yor sheet. Let $\Psi(z)=\Gamma'(z)/\Gamma(z)$ be the digamma function. For $\kappa>0$, define
\begin{align*}
\theta(\kappa)=(\Psi')^{-1}(\kappa),\ f(\kappa)=\theta(\kappa)\Psi'(\theta(\kappa))-\Psi(\theta(\kappa)), c(\kappa)=(-2^{-1}\Psi''(\theta(\kappa)))^{1/3}.
\end{align*}
Set
\begin{align*}
a_1=c(\kappa)^{-1},\ a_2=2c(\kappa)^2,\ a_3=\kappa,\ a_4=-2\theta(\kappa)c(\kappa), a_5=-c(\kappa)^{-1}f(\kappa),
\end{align*}
and define
\begin{align*}
\cS^{\textup{OY}\ar\textup{FP}}_n(x,y):=a_1 n^{-1/3}B((a_2n^{2/3}x,n)\Rightarrow (a_2 n^{2/3}y+a_3n,1))+a_4n^{1/3}(y-x)+a_5n^{2/3}.
\end{align*} 

The following theorem uses the upcoming tightness result  \cite{wu2025interior}.

\begin{theorem}\label{t:OY_to_FP}
As $C(\mathbb{R}\times\mathbb{R},\mathbb{R})$-valued random variables, $\cS^{\textup{OY}\ar\textup{FP}}_n$ converges in distribution to the Airy sheet $\cS$.
\end{theorem}

\begin{proof}
Due to the translation symmetry of both $\cS^{\textup{OY}\ar\textup{FP}}_n$ and $\cS$, it suffices to prove this convergence when both are restricted to $[0,\infty)\times \mathbb{R}$.  

Let $(TB)_n:(0,\infty)\times \{1,2,\dots, n\}$ be the O'Connell-Yor line ensemble constructed from $B|_{[0,\infty)\times \{1,2,\dots, n\}}$ \cite{o'connell2012DP} such that
\begin{align*}
(TB)_n(y,1)=B((0,n)\Rightarrow (y,1))\ \textup{for all}\ y> 0.
\end{align*}
Define
\begin{equation}\label{def:OY_action_rep}
Z^{\textup{OY}}_n(x,y,k)=\left\{ \begin{array}{cc}
(TB)_n(y,k), & x=0,\\
(TB)_n((x,n)\Rightarrow (y,k)), & x>0.
\end{array} \right.
\end{equation}
From \cite{noumi2002tropical,corwin2020invariance}, for any $y\geq x\geq 0$,
\begin{align}\label{equ:OY_to_FP_1}
Z^{\textup{OY}}_n(x,y,1)=B((x,n)\Rightarrow (y,1)).
\end{align}
From \cite{wu2023kpz}, $Z^{\textup{OY}}_n(\cdot,\cdot,1)$ satisfies the quadrangle inequality. Hence $Z^{\textup{OY}}_n$ is an $e$-polymer action representation. Define
\begin{align*}
Z^{\textup{OY}\ar \textup{FP}}_n(x,y,k)=a_1 n^{-1/3}Z^{\textup{OY}}_n(a_2 n^{2/3}x,a_2n^{2/3}y+a_3n,k)+a_4 n^{1/3}(y-x)+a_5n^{2/3}\\
+(k-1)a_1 n^{-1/3}\log(a_2 n^{2/3}).
\end{align*} 
From Lemma~\ref{lem:action_representation_COV}, $Z^{\textup{OY}\ar \textup{FP}}_n$ is an $\exp(a_1^{-1}n^{1/3})$-polymer action representation. From \eqref{equ:OY_to_FP_1},
$$
Z^{\textup{OY}\ar \textup{FP}}_n(x,y,1)=S^{\textup{OY}\ar \textup{FP}}_n(x,y).
$$

Now we apply Theorem~\ref{thm:action-rep-sym}. The symmetry assumptions (Sym1) and (Sym2) follow from \eqref{equ:OY_to_FP_1}. From \cite{wu2025interior}, $Z_n^{\textup{OY}\ar\textup{FP}}(0,\cdot,\cdot)$ converges in distribution to the parabolic Airy line ensemble. This verifies the assumption (A1). Then the desired coupling is constructed from Theorem~\ref{thm:action-rep-sym}. The proof is complete. 
\end{proof}

\subsection{Log-gamma polymers}

Given $(x_1,k_1)$, $(x_2,k_2)\in\mathbb{Z}\times\mathbb{N}$ with $(x_1,k_1)\preceq (x_2,k_2)$, we denote by $\mathcal{P}_d((x_1,k_1)\ar (x_2,k_2))$ the collection of directed discrete paths from $(x_1,k_1)$ to $(x_2,k_2)$. That is,
\begin{align*}
\mathcal{P}_d((x_1,k_1)\ar (x_2,k_2))=\left\{\gamma\cap (\mathbb{Z}\times\mathbb{N})\, :\, \gamma\in  \mathcal{P}((x_1,k_1)\ar (x_2,k_2))\right\}.
\end{align*}
Let $A$ be a discrete environment defined on a subset of $\mathbb{Z}\times\mathbb{N}$ which contains $([x_1-1,x_2]\cap\mathbb{Z}) \times [k_2,k_1]$. Given $\gamma_d\in \mathcal{P}_d((x_1,k_1)\ar (x_2,k_2))$, we define
\begin{align*}
L(\gamma_d,A)=\sum_{(x,k)\in\gamma_d} A(x,k)-A(x-1,k).
\end{align*}

\begin{align}
A((x_1,k_1)\Rightarrow_d (x_2,k_1))=\log \sum_{\gamma_d\in \mathcal{P}_d((x_1,k_1)\ar (x_2,k_2))}  e^{L(\gamma_d,A)}.
\end{align}

Recall that a continuous random variable $X$ is said to have the inverse-gamma distribution with parameter $\theta>0$ if its density is given by
\begin{align}\label{equ:log_gamma_density}
f_\theta(x)=\frac{\mathbbm{1}(x>0)}{\Gamma(\theta)}\cdot x^{-\theta-1}\cdot\exp(-x^{-1}).
\end{align}
Let $\{D(x,k)\, :\, (x,k)\in\mathbb{Z}\times\mathbb{N}\}$ be a collection of random variables such that $D(0,k)\equiv 0$ and $\{D(x,k)-D(x-1,k)\, :\, (x,k)\in\mathbb{Z}\times\mathbb{N}\}$ are i.i.d. random variables with density $f_\theta$ as in \eqref{equ:log_gamma_density}. 

Recall that $\Psi(x)$ is the digamma function. Denote $\sigma=\sqrt{\Psi'(\theta/2)}$ and let $h_\theta,\ d_\theta$ be the functions defined as in \cite[(1.8),(1.9)]{barraquand2023tightness} respectively. Set
$$a_1=d_\theta(1)^{-1},a_2=2\sigma^{-2}d_\theta(1)^2,a_3=1,a_4=-2\sigma^{-2}d_\theta(1)h'_\theta(1),a_5=d_\theta(1)^{-1}h_\theta(1).$$
For $x,y\in (a_2n^{2/3})^{-1}\mathbb{Z}$ with $y\geq x-a_2^{-1}a_3n^{1/3}+a_2^{-1}n^{-2/3}$, define
\begin{align*}
\cS^{\textup{LG}\ar\textup{FP}}_n(x,y):=a_1 n^{-1/3}D((a_2n^{2/3}x+1,n)\Rightarrow_d (a_2 n^{2/3}y+a_3n ,1))+a_4n^{1/3}(y-x)+a_5n^{2/3}.
\end{align*} 
Extend $\cS^{\textup{LG}\ar\textup{FP}}_n$ to $\overline{\cS}^{\textup{LG}\ar\textup{FP}}_n$ through linear interpolation.

The following result was originally proved in \cite{zhang2025convergence}.
 
\begin{theorem}\label{t:LG_to_FP}
As $C(\mathbb{R}\times\mathbb{R},\mathbb{R})$-valued random variables, $\overline{\cS}^{\textup{LG}\ar\textup{FP}}_n$ converges in distribution to the Airy sheet $\cS$.
\end{theorem}
\begin{proof}
Due to the translation symmetry of both $\cS^{\textup{LG}\ar\textup{FP}}_n$ and $\cS$, it suffices to prove this convergence when both are restricted to $[0,\infty)\times \mathbb{R}$.

For $n\in\mathbb{N}$, let $\{\tilde{D}_n(x,k)\, :\, 1\leq k\leq n\wedge x\}$ be the log-gamma line ensemble constructed from $D(x,k),\ x\in \mathbb{N}, 1\leq k\leq n$ \cite{corwin2014tropical} such that
\begin{align*}
\tilde{D}_n(y,1)=D((1,n)\Rightarrow_d (y,1))\ \textup{for all}\ y\geq 1.
\end{align*} 
We extend $\tilde{D}_n$ to the domain $(x,k)\in  (\mathbb{N}\cup\{0\})\times \{1,2,\dots, n\} $ by setting $\tilde{D}_n(x,k)=0$ if $0\leq x\leq k-1$. Define $Z^{\textup{LG}}_n(x,y,k)$ on $0\leq x\leq y-1\in\mathbb{Z}$ and $1\leq k\leq n$ by 
\begin{align*}
Z^{\textup{LG}}_n(x,y,k)=\left\{
\begin{array}{cc}
\tilde{D}_n(y,k), & x=0,\\
\tilde{D}((x+1,n\wedge (x+1))\Rightarrow_d (y,k)), & x\geq 1,   (x+1)\geq k,\\
-\infty, & x\geq 1, (x+1)<k.
\end{array}
\right.
\end{align*}
From \cite{noumi2002tropical,corwin2020invariance}, for any $0\leq x\leq y-1$,
\begin{align}\label{equ:LG_to_FP_1}
Z^{\textup{LG}}_n(x,y,1)=D((x+1,n)\Rightarrow_d (y,1)).
\end{align}
Define
\begin{align*}
Z^{\textup{LG}\ar \textup{FP}}_n(x,y,k)= a_1 n^{-1/3}Z^{\textup{LG}}_n(a_2 n^{2/3}x,a_2n^{2/3}y+a_3n,k)+a_4 n^{1/3}(y-x)+a_5n^{2/3}. 
\end{align*}
It is direct to check that $Z^{\textup{LG}\ar \textup{FP}}_n$ is a $(a_2n^{2/3})^{-1}$-discrete, $\exp(a_1^{-1}n^{1/3})$-polymer action representation. From \eqref{equ:LG_to_FP_1},
$$
Z^{\textup{LG}\ar \textup{FP}}_n(x,y,1)=S^{\textup{LG}\ar \textup{FP}}_n(x,y).
$$

Now we apply Theorem~\ref{t:action-rep_discrete}. The symmetry assumptions (Symd1) and (Symd2) follow from \eqref{equ:LG_to_FP_1}. From \cite[Corollary 25.2]{aggarwal2025strong}, $Z_n^{\textup{LG}\ar\textup{FP}}(0,\cdot,\cdot)$ converges in distribution to the parabolic Airy line ensemble. This verifies the assumption (Ad1). Then the desired coupling is constructed from Theorem~\ref{t:action-rep_discrete}. The proof is complete. 
\end{proof}
\subsection{The KPZ equation}

For $t>0$, let $\mathcal{H}_t:\mathbb{R}\times\mathbb{N}\to\mathbb{R}$ be the $\textup{KPZ}_t$ line ensemble and $\mathcal{S}^{\textup{KPZ}}_t:\mathbb{R}\times\mathbb{R}\to\mathbb{R}$ be the KPZ sheet. In the next theorem, we represent the KPZ sheet as a polymer action in the KPZ line ensemble.  

\begin{theorem}\label{thm:KPZ_action}
 Fix $t>0$. There exists a random function $Z^{\textup{KPZ}}_t$ such that 
 \begin{itemize}
 \item $Z^{\textup{KPZ}}_t(0,\cdot,\cdot)$ agrees with the $\textup{KPZ}_t$ line ensemble $\mathcal{H}_t(\cdot,\cdot)$.
 \item $Z^{\textup{KPZ}}_t$ is an $e$-polymer representation.
 \item Almost surely $Z^{\textup{KPZ}}_t(x,y,1)=\cS_t^{\textup{KPZ}}(x,y)$ for all $(x,y)\in\mathbb{R}^2$. 
 \end{itemize}
\end{theorem}

\begin{proof}
Fix $t>0$. Let 
\begin{align*}
 C_1(t,n)=& n^{1/2}t^{-1/2}+2^{-1},\ C_{3,k}(t)= -(k-1)\log t+\log(k-1)!, \\
C_2(t,n)= & n+2^{-1}n^{1/2}t^{1/2}- (n-1)\log( n^{1/2}t^{-1/2}) .
\end{align*}
Recall that $Z^{\textup{OY}}$ is defined in \eqref{def:OY_action_rep}. Define
\begin{align*}
Z^{\textup{OY}\ar\textup{KPZ}}_{t,n}(x,y,k)=\left\{
\begin{array}{cc}
Z^{\textup{OY}}_n(0,y+t^{1/2}n^{1/2},k)-C_1(t,n)y-C_2(t,n)-C_{3,k}(t,n), & x=0,\\
Z^{\textup{OY}}_n(x,y+t^{1/2}n^{1/2},k)-C_1(t,n)(y-x)-C_2(t,n)-C_{3,1}(t,n), & x>0.
\end{array}
\right.
\end{align*}
From Lemma~\ref{lem:action_representation_COV}, $Z^{\textup{OY}\ar\textup{KPZ}}_{t,n}$ is an $e$-polymer action representation.  From \cite{corwin2016kpz}, $Z^{\textup{OY}\ar\textup{KPZ}}_{t,n}(0,\cdot,\cdot)$ converges in distribution to the $\textup{KPZ}_t$ line ensemble $\mathcal{H}_t(\cdot,\cdot)$. From \cite{nica2021intermediate},
 $Z^{\textup{OY}\ar\textup{KPZ}}_{t,n}(x,y,1)$ converges in distribution to the KPZ sheet $\cS_t^{\textup{KPZ}}(x,y)$. From the Skorokhod representation, we can have a coupling such that these convergence holds almost surely.  
 
Now we construct $Z^{\textup{KPZ}}_t$. For $x=0$, we set
\begin{align*}
Z^{\textup{KPZ}}_t(0,y,k)=\mathcal{H}_t(y,k)\ \textup{for all}\ (y,k)\in\mathbb{R}\times\mathbb{N}.
\end{align*}
Fix a realization such that the above convergence holds. Recall that for all $x>0$, we have that $Z^{\textup{OY}\ar\textup{KPZ}}_{t,n}(x,\cdot,\cdot)$ is an $e$-polymer action with respect to $Z^{\textup{OY}\ar\textup{KPZ}}_{t,n}(0,\cdot,\cdot)$. Moreover,  $Z^{\textup{OY}\ar\textup{KPZ}}_{t,n}(x,y,1)$ converges to $\cS_t^{\textup{KPZ}}(x,y)$. Proposition~\ref{p:polymer-to-polymer} implies that $Z^{\textup{OY}\ar\textup{KPZ}}_{t,n}(x,\cdot,\cdot)$ converges pointwise. Denote the limit by $Z^{\textup{KPZ}}_{t}(x,\cdot,\cdot)$. Then $Z^{\textup{KPZ}}_{t}(x,\cdot,\cdot)$ is an $e$-polymer action with respect to $Z^{\textup{KPZ}}_{t}(0,\cdot,\cdot)$. Moreover, $Z^{\textup{KPZ}}_{t}(x,y,1)=\cS_t^{\textup{KPZ}}(x,y)$.
\end{proof}
Define 
\begin{align*}
\cS^{\textup{KPZ}\ar \textup{FP}}_t(x,y)=2^{1/3}t^{-1/3}\cS^{\textup{KPZ}}_t(2^{1/3}t^{2/3}x,2^{1/3}t^{2/3}y)+2^{1/3}t^{2/3}/4!.
\end{align*}
 
The next theorem shows that $\cS^{\textup{KPZ}\to\textup{FP}}_t$ converges to the Airy sheet $\cS$, a result originally proved in \cite{wu2023kpz}.

\begin{theorem}
As $C(\mathbb{R}\times\mathbb{R},\mathbb{R})$-valued random variables, $\cS^{\textup{KPZ}\ar \textup{FP}}_t$ converges in distribution to the Airy sheet $\cS$.
\end{theorem}

\begin{proof}
Due to translation symmetry, it suffices to prove convergence when both $\cS^{\textup{KPZ}\ar \textup{FP}}_t$ and $\cS$ are restricted to $[0,\infty)\times\mathbb{R}$.

Fix an arbitrary sequence $t\uparrow\infty$. We will show that there exists a coupling of $\cS^{\textup{KPZ}\ar \textup{FP}}_t$ and $\cS|_{[0,\infty)\times\mathbb{R}}$ such that almost surely $\cS^{\textup{KPZ}\ar \textup{FP}}_t|_{[0,\infty)\times\mathbb{R}}$ converges locally uniformly to $\cS|_{[0,\infty)\times\mathbb{R}}$.

Let $Z^{\textup{KPZ}}_t$ be given by Theorem~\ref{thm:KPZ_action}. Set 
\begin{align*}&Z^{\textup{KPZ}\ar\textup{FP}}_t(x,y,k)\\ &=2^{1/3}t^{-1/3}Z^{\textup{KPZ}}_t(2^{1/3}t^{2/3}x,2^{1/3}t^{2/3}y,k)+2^{1/3}t^{2/3}/4!+ (k-1) 2^{1/3}t^{-1/3}\log(2^{1/3}t^{2/3}).
\end{align*}
From Lemma~\ref{lem:action_representation_COV}, $Z^{\textup{KPZ}\ar\textup{FP}}_t$ is a $\exp((2^{-1}t)^{1/3})$-polymer representation.  Now we apply Theorem~\ref{thm:action-rep-sym}. The symmetry assumptions (Sym1) and (Sym2) follow from Theorem~\ref{thm:KPZ_action}. From \cite{virag2020heat}, \cite{dimitrov2021characterization}, \cite{wu2023convergence}, and \cite{aggarwal2025strong}, $Z^{\textup{KPZ}\ar\textup{FP}}_t(0,\cdot,\cdot)$ converges in distribution to the parabolic Airy line ensemble. This verifies the assumption (A1). Then the desired coupling is constructed from Theorem~\ref{thm:action-rep-sym}. The proof is complete.
\end{proof}
\begin{appendix}
\section{Hypograph and Skorokhod topology}
\subsection{Hypograph topology}\label{a:UC}
We equip the extended real numbers, $\mathbb{R}\cup\{\pm\infty\}$, with the topology that compactifies $\mathbb{R}$ at $\pm\infty$.  Denote by $\uc(\mathbb{R})$ the collection of upper semicontinuous functions from $\mathbb{R}$ to $\mathbb{R}\cup\{\pm \infty\}$. We equip $\uc(\mathbb{R})$ with the {\bf local hypograph topology}, which is described below. 

Given a function $f:\mathbb{R}\to\mathbb{R}\cup\{\pm\infty\}$, the hypograph of $f$, denoted by $\textup{hypo}(f)$, is a subset of $\mathbb{R}\times (\mathbb{R}\cup\{\pm\infty\})$ defined as 
\begin{align*}
\textup{hypo}(f)=\{(x,y)\in \mathbb{R}\times (\mathbb{R}\cup\{\pm\infty\})\, :\, y\leq f(x)\}.
\end{align*}
It is direct to check that $f\in\uc(\mathbb{R})$ if and only if $\textup{hypo}(f)$ is a closed subset of $\mathbb{R}\times (\mathbb{R}\cup\{\pm\infty\})$. Consider a map $\Phi:\mathbb{R}\times (\mathbb{R}\cup\{\pm\infty\})\to \mathbb{R}^2$ given by
\begin{align*}
\Phi(x,y)=\left( \arctan(x), \frac{2\arctan(y)}{\pi(1+x^2)} \right).
\end{align*} 
Here we adopt the convention that $\arctan(\pm\infty)=\pm\frac{\pi}{2}$. The image of $\Phi$ is given by
\begin{align*}
 \left\{(z,w)\in\mathbb{R}^2\, :\, |z|< \frac{\pi}{2}\ \textup{and}\ |w|\leq \cos^2 z \right\}.
\end{align*}   
Adding $(\pm\frac{\pi}{2},0)$ to the image of $\Phi$, we obtain a compact set $K$. Denote by $\textup{CL}(K)$ the collection of closed subsets of $K$. Consider an injective map from $\uc(\mathbb{R})$ to $\textup{CL}(K)$ given by $f\mapsto \Phi(\textup{hypo}(f))\cup\{(\pm\frac{\pi}{2},0)\}.$ Through this map, $\uc(\mathbb{R})$ is identified with a subset of $\textup{CL}(K)$. The local hypograph topology of $\uc(\mathbb{R})$ is the subset topology of the topology on $\textup{CL}(K)$ induced by the Hausdorff distance.

For any $\alpha\in\mathbb{R}$, we denote by $\uc([\alpha,\infty))$ the collection of upper semicontinuous functions from $[\alpha,\infty)$ to $\mathbb{R}\cup\{\pm\infty\}.$ We embed $\uc([\alpha,\infty))$ into $\uc(\mathbb{R})$ by
\begin{align*}
f\mapsto \bar{f}(x)=\left\{
\begin{array}{cc}
f(x), & x\geq \alpha,\\
f(\alpha), & x<\alpha.
\end{array}
\right.
\end{align*}
Denote
\begin{align*}
\uc(\mathbb{R}\times\mathbb{N})=\{(f_1,f_2,\dots)\, :\, f_i\in\uc(\mathbb{R})\ \textup{for all}\ i\in\mathbb{N}\}.
\end{align*}
We equip $\uc(\mathbb{R}\times\mathbb{N})$ with the product topology through identifying it with a countable product of $\uc(\mathbb{R})$'s.

\begin{lemma}\label{lem:UC_compact}
Under the local hypograph topology, $\uc(\mathbb{R})$ is metrizable and compact.
\end{lemma}
\begin{proof}
The metrizability is clear, as the topology is induced by the Hausdorff distance. From \cite[Theorem 7.3.8]{burago2001course}, $\textup{CL}(K)$ is compact under the Hausdorff distance. It is direct to check that the image of $\uc(\mathbb{R})$ is a closed subset of $\textup{CL}(K)$. Consequently, $\uc(\mathbb{R})$ is compact.
\end{proof}

\begin{lemma}
Let $f_n$ be a sequence in $\uc(\mathbb{R})$ and $f\in \uc(\mathbb{R})$. Suppose $f_n$ converges to $f$ in $\uc(\mathbb{R})$. Then the following statements hold:
\begin{itemize}
\item for any bounded closed set $F\subset \mathbb{R}$,    
\begin{align}\label{equ:UC_limit_closed}
\limsup_{n\to\infty} \sup_{x\in F}f_n(x)\leq \sup_{x\in F}f(x); 
\end{align}
\item for any bounded open set $U\subset\mathbb{R}$,   
\begin{align}\label{equ:UC_limit_open}
\liminf_{n\to\infty} \sup_{x\in U}f_n(x)\geq \sup_{x\in U}f(x). 
\end{align}
\end{itemize}
\end{lemma}
\begin{remark}
Conditions \eqref{equ:UC_limit_closed} and \eqref{equ:UC_limit_open} are sufficient to show the convergence of $f_n$ to $f$ in $\uc(\mathbb{R})$. We do not present the proof because this direction is not needed for our purposes.
\end{remark}
\begin{proof}
Let $f_n$ and $f$ be given as in the lemma. Define $g_n$ and $g$ on $(-\frac{\pi}{2},\frac{\pi}{2})$ by
\begin{align*}
g_n(x)=\frac{2 \arctan(f_n(\tan x))}{\pi} ,\ g(x)=\frac{2\arctan(f(\tan x))}{\pi} .
\end{align*}
Define
\begin{align*}
\overline{\textup{hypo}}(g_n)&=\left\{ (x,y)\in\mathbb{R}^2\, :\, x\in(-\frac{\pi}{2},\frac{\pi}{2}),\ -\cos^2x\leq y\leq g_n(x)\cos^2x\right\}\cup \{(\pm\frac{\pi}{2},0)\},\\
\overline{\textup{hypo}}(g)&=\left\{ (x,y)\in\mathbb{R}^2\, :\, x\in(-\frac{\pi}{2},\frac{\pi}{2}),\ - \cos^2x\leq y\leq g(x)\cos^2x\right\}\cup \{(\pm\frac{\pi}{2},0)\}.
\end{align*}
The convergence of $f_n$ to $f$ in $\uc(\mathbb{R})$ is equivalent to the convergence of $\overline{\textup{hypo}}(g_n)$ to $\overline{\textup{hypo}}(g)$ in the Hausdorff distance. We aim to show that for any closed set $F\subset (-\frac{\pi}{2},\frac{\pi}{2})$ and any relatively compact open set $U\Subset (-\frac{\pi}{2},\frac{\pi}{2})$, we have
\begin{equation}\label{equ:UC_closed_set}
\limsup_{n\to\infty} \sup_{x\in F}g_n(x)\leq \sup_{x\in F}g(x),
\end{equation} 
and
\begin{equation}\label{equ:UC_open_set}
\liminf_{n\to\infty} \sup_{x\in U}g_n(x)\geq \sup_{x\in U}g(x).
\end{equation} 

Fix a closed set $F\subset (-\frac{\pi}{2},\frac{\pi}{2})$. The upper semicontinuity of $g_n$ implies there exists $x_n\in F$ such that $g_n(x_n)=\sup_{x\in F}g_n(x)$. Since $\overline{\textup{hypo}}(g_n)$ converges to $\overline{\textup{hypo}}(g)$ in the Hausdorff distance, there exists $(x'_n,y'_n)\in \overline{\textup{hypo}}(g)$ such that $\|(x_n,g_n(x_n)\cos^2 x_n)-(x'_n,y'_n) \|\to 0$. Therefore,
\begin{align*}
\limsup_{n\to\infty} \sup_{x\in F}g_n(x)&=\limsup_{n\to\infty}g_n(x_n)=\limsup_{n\to\infty} \frac{y'_n}{\cos^2 x_n}\\
&=\limsup_{n\to\infty} \frac{y'_n}{\cos^2 x'_n} \leq \limsup_{n\to\infty}g(x'_n)\leq \sup_{x\in F}g(x).
\end{align*} 
We used the upper semicontinuity of $g$ in the last inequality. This proves \eqref{equ:UC_closed_set}.

Next, fix an open set $U\Subset (-\frac{\pi}{2},\frac{\pi}{2})$. For any $\varepsilon>0$, there exists $x_0\in U$ such that $g(x_0)\geq \sup_{x\in U} g(x)-\varepsilon$. Since $\overline{\textup{hypo}}(g_n)$ converges to $\overline{\textup{hypo}}(g)$ in the Hausdorff distance, there exists $(x_n,y_n)\in \overline{\textup{hypo}}(g_n)$ such that $\|(x_0,g(x_0)\cos^2 x_0)-(x_n,y_n)\|\to 0$. Therefore,
\begin{align*}
\sup_{x\in U} g(x)&\leq g(x_0)+\varepsilon= \lim_{n\to\infty} \frac{y_n}{\cos^2 x_0}+\varepsilon=\lim_{n\to\infty} \frac{y_n}{\cos^2 x_n}+\varepsilon\\
&\leq \liminf_{n\to\infty} g_n(x_n)+\varepsilon\leq \liminf_{n\to\infty} \sup_{x\in U} g_n(x)+\varepsilon.
\end{align*}
Since $\varepsilon$ is arbitrary, we prove \eqref{equ:UC_open_set}. The proof is complete.
\end{proof}

\begin{lemma}\label{lem:UC_no_jump}
Let $f_n$ be a sequence in $\uc(\mathbb{R})$ and $f\in \uc(\mathbb{R})$. Assume 
\begin{itemize}
\item each $f_n$ is monotone nondecreasing;
\item $f_n$ converges to $f$ in $\uc(\mathbb{R})$;
\item for all $x\in\mathbb{Q}$, $\lim_{n\to\infty} f_n(x)=:\tilde{f}(x)$ exists.
\end{itemize}
Further assume for some $x_0\in\mathbb{Q}$, 
\begin{align*}
\lim_{x\in\mathbb{Q}, x\downarrow x_0}\tilde{f}(x)=\tilde{f}(x_0).
\end{align*}
Then $f(x_0)=\tilde{f}(x_0)$.
\end{lemma}
\begin{proof}
From \eqref{equ:UC_limit_closed}, we have $\tilde{f}(x_0)\leq f(x_0)$. To prove $f(x_0)\leq \tilde{f}(x_0)$, it suffices to show that $f(x_0)\leq \tilde{f}(x)$ for all $x\in \mathbb{Q}$ with $x>x_0$. Fix $x\in \mathbb{Q}$ with $x>x_0$ and set $\varepsilon=x-x_0$. From \eqref{equ:UC_limit_open} and the monotonicity of $f_n$,
\begin{align*}
f(x_0)\leq \sup_{y\in (x_0-\varepsilon,x_0+\varepsilon)}f(y)\leq \liminf_{n\to\infty} \sup_{y\in (x_0-\varepsilon,x_0+\varepsilon)}f_n(y)\leq  \liminf_{n\to\infty} f_n(x)=\tilde{f}(x) .
\end{align*}  
This completes the proof.
\end{proof}
\subsection{Skorokhod topology}
We denote by $\mathbb{D}(\mathbb{R})$ the space of cadlag functions defined on $\mathbb{R}$ equipped with the Skorokhod topology. We refer the reader to Chapter 3 of \cite{billingsley1999convergence} for more details. Denote
\begin{align*}
\mathbb{D}(\mathbb{R}\times\mathbb{N})=\{(f_1,f_2,\dots)\, :\, f_i\in\mathbb{D}(\mathbb{R})\ \textup{for all}\ i\in\mathbb{N}\}.
\end{align*}
We equip $\mathbb{D}(\mathbb{R}\times\mathbb{N})$ with the product topology by identifying it with a countable product of $\mathbb{D}(\mathbb{R})$'s.

\begin{lemma}\label{lem:Skorokhod}
Let $f_n$ be a sequence in $\mathbb{D}(\mathbb{R})$, $f\in \mathbb{D}(\mathbb{R})$, and $\tilde{f}\in C(\mathbb{R},\mathbb{R})$. Suppose that $f_n$ converges to $f$ in $\mathbb{D}(\mathbb{R})$, and that $\lim_{n\to\infty}f_n(x)=\tilde{f}(x)$ for all $x\in\mathbb{Q}$. Then $f(y)=\tilde{f}(y)$ for all $y\in\mathbb{R}$.
\end{lemma} 
\begin{proof}
We claim that for all $x_0\in\mathbb{Q}$, we have 
\begin{equation}\label{equ:Skorokhod_1}
f(x_0)=\tilde{f}(x_0)\ \textup{or}\ f(x_0-)=\tilde{f}(x_0).
\end{equation}
Assume for a moment \eqref{equ:Skorokhod_1} holds. Fix arbitrary $y\in\mathbb{R}$ and let $y_j$ be a sequence of rational numbers such that $y_j\downarrow y$. From \eqref{equ:Skorokhod_1}, there exists $y'_j\in (y,y_j]$ such that $|f(y'_j)-\tilde{f}(y_j)|\leq j^{-1}$. Therefore,
\begin{align*}
f(y)=\lim_{j\to\infty} f(y'_j)=\lim_{j\to\infty}\tilde{f}(y_j)=\tilde{f}(y).
\end{align*} 

It remains to prove \eqref{equ:Skorokhod_1}. Since $f_n$ converges to $f$ in $\mathbb{D}(\mathbb{R})$, there exists a sequence of strictly increasing, continuous maps from $\mathbb{R}$ to $\mathbb{R}$, denoted by $\lambda_n$, such that
$$
\lim_{n\to\infty} \sup_{y\in\mathbb{R}}|\lambda_n(y)-y|=0,
$$
and 
$$\lim_{n\to\infty}\sup_{y\in [-r,r]}|f_n(\lambda_n(y))-f(y)|=0$$
for all $r>0$. Fix $x_0\in\mathbb{Q}$ and let $x_n=\lambda^{-1}_n(x_0)$. Then $x_n\to x_0$ and $|f_n(x_0)-f(x_n)|\to 0$. Therefore,
\begin{align*}
\tilde{f}(x_0)=\lim_{n\to\infty} f_n(x_0)=\lim_{n\to\infty} f(x_n).
\end{align*}  
If $x_n\geq x_0$ for infinitely many $n$, then we have $\lim_{n\to\infty} f(x_n)=f(x_0)$. If $x_n<x_0$ for $n$ large enough, then $\lim_{n\to\infty} f(x_n)=f(x_0-)$. This shows \eqref{equ:Skorokhod_1} and completes the proof. 
\end{proof}

\bigskip
\noindent {\bf Acknowledgments.}
B.V. was supported by the Canada Research Chair program, the NSERC Discovery Accelerator grant, and the MTA Momentum Random Spectra research group. X.W. was partially supported by the NSF through NSF-2348188 and by the Simons Foundation through MPS-TSM-00007939.
\end{appendix}

\bibliographystyle{dcu}
\bibliography{asl}

\end{document}